\documentclass[10pt, reqno]{amsart}
\usepackage{lipsum}
\usepackage{a4wide}
\usepackage{amssymb}
\usepackage{amsmath}
\usepackage{amsthm}
\usepackage{tikz}
\usepackage{tikz-cd}
\usepackage{amscd}
\usepackage[all]{xy}
\usepackage{hyperref}
\usepackage{enumitem}
\hypersetup{colorlinks,linkcolor={blue},citecolor={blue},urlcolor={red}}
\usepackage{color}
\theoremstyle{plain}
\usepackage[margin=2.5cm]{geometry}
\numberwithin{equation}{section}

\newcommand{\id}{\operatorname{id}}

\newcommand{\sch}[1]{\operatorname{{\bf #1}}}

\newcommand{\ind}{\operatorname{ind}}

\newcommand{\res}{\operatorname{res}}

\newtheorem{theorem}{Theorem}[section]
\newtheorem{corollary}[theorem]{Corollary}
\newtheorem{lemma}[theorem]{Lemma}
\newtheorem{remark}[theorem]{Remark}
\newtheorem{proposition}[theorem]{Proposition}

 \newtheorem{claim}{Claim}

\title[Tate cohomology of Whittaker lattices]{Tate cohomology of
Whittaker lattices and Base change of generic representations of
${\rm GL}_n$} 
\author{Santosh Nadimpalli and Sabyasachi Dhar}
\begin{document}
\begin{abstract}
  Let $p$ and $l$ be two distinct odd primes, and let $n\geq 2$ be a
  positive integer. Let $E$ be a finite Galois extension of degree $l$
  of a $p$-adic field $F$. Let $q$ be the cardinality of the residue
  field of $F$.  Let $\pi_F$ be an integral $l$-adic generic
  representation of ${\rm GL}_n(F)$, and let $\pi_E$ be the base
  change of $\pi_F$. Let $J_l(\pi_F)$ (resp. $J_l(\pi_E)$) be the
  unique generic component of the mod-$l$ reduction $r_l(\pi_F)$
  (resp. $r_l(\pi_E)$). Assuming that $l$ does not divide
  $|{\rm GL}_{n-1}(\mathbb{F}_q)|$, we prove that the Frobenius twist
  of $J_l(\pi_F)$ is the unique generic subquotient of the Tate
  cohomology group
  $\widehat{H}^0({\rm Gal}(E/F), J_l(\pi_E))$--considered as a
  representation of ${\rm GL}_n(F)$.
\end{abstract}
\maketitle
\section{Introduction}\label{section_1}
Let $l$ be a prime number, and let $F$ be a number field. Let
$\sch{G}$ be a reductive algebraic group defined over $F$, and let
$\sigma$ be an automorphism of order $l$ of $\sch{G}$. D.Treumann and
A.Venkatesh have constructed a functorial lift of a mod-$l$
automorphic form for $\sch{G}^\sigma$ to a mod-$l$ automorphic form
for $\sch{G}$ (see \cite{MR3432583}). They conjectured that the
mod-$l$ local functoriality at ramified places must be realised in
Tate cohomology, and they defined the notion of linkage (see
\cite[Section 6.3]{MR3432583} for more details). Among many
applications of this set up, we focus on local base change lifting
from $\sch{G}^\sigma=\sch{{\rm GL}}_n/F$ to
$\sch{G}={\rm Res}_{E/F}\sch{\rm GL}_n/E$, where $E/F$ is a Galois
extension of $p$-adic fields with $[E:F]=l$. Truemann and Venkatesh's
conjecture on linkage in Tate cohomology is verified for local base
change of depth-zero cuspidal representations by N.Ronchetti, and a
precise conjecture in the context of local base change of $l$-adic
higher depth cuspidal representations was formulated in
\cite[Conjecture 2]{MR3551160}. In this article, using Whittaker
models and Rankin-Selberg zeta functions, we prove this conjecture for
${\rm GL}_n$ under the assumption that $l$ does not divide the
pro-order of ${\rm GL}_{n-1}(F)$ whenever $n>2$. In fact, when the
prime $l$ does not divide the pro-order of ${\rm GL}_{n-1}(F)$, we
prove a much stronger theorem that the Frobenius twist of a mod-$l$
generic representation of ${\rm GL}_{n}(F)$ occurs as a sub-quotient
of the zeroth Tate cohomology of its base change lifting to
${\rm GL}_n(E)$ (see Theorem \ref{noncuspidal_n_thm} for the precise
result).
	
Let us introduce some notations to state the results of this
article. From now, we assume that $F$ is a finite extension of
$\mathbb{Q}_p$ with residue field $\mathbb{F}_q$. Let $E$ be a finite
Galois extension of $F$ with $[E: F]=l$, where $l$ and $p$ are
distinct odd primes. Let $\pi_F$ be an integral $l$-adic generic
representation of ${\rm GL}_n(F)$. The mod-$l$-reduction of $\pi_F$
has a unique generic component and it is denoted by $J_l(\pi_F)$ (see
\cite[Section 1.8.4]{MR1821157}). The base change lift of $\pi_F$ to
${\rm GL}_n(E)$ is denoted by $\pi_E$ (for the definition, see
subsection (\ref{base_change})). Note that $\pi_E$ is also an integral
$l$-adic generic representation of ${\rm GL}_n(E)$. Moreover, there is
an isomorphism $T:\pi_E\xrightarrow{\sim} \pi_E^\gamma$, where
$\pi_E^\gamma$ is the twist of $\pi_E$ by a generator $\gamma$ of
${\rm Gal}(E/F)$. Then the unique generic component $J_l(\pi_E)$ of
the mod-$l$ reduction $r_l(\pi_E)$ is also stable under the action of
${\rm Gal}(E/F)$--induced by $T$. In this article, Tate cohomology
groups are always with respect to the action of ${\rm Gal}(E/F)$. We
prove the following theorem:
\begin{theorem}\label{intro_n_thm}
Let $F$ be a finite extension of $\mathbb{Q}_p$, and let $E$ be a
finite Galois extension of $F$ with $[E:F] =l$, where $p$ and $l$
are distinct odd primes such that that $l$ does not divide
$|{\rm GL}_{n-1}(\mathbb{F}_q)|$. Let $\pi_F$ be an integral
$l$-adic generic representation of ${\rm GL}_n(F)$, and let $\pi_E$
be the base change lifting of $\pi_F$ to ${\rm GL}_n(E)$. Then, the
Frobenius twist of $J_l(\pi_F)$ occurs a subquotient of the zeroth
Tate cohomology group $\widehat{H}^0 (J_l(\pi_E))$, considered as a
representation of ${\rm GL}_n(F)$.
\end{theorem}
We note some immediate remarks on the hypothesis in Theorem
\ref{intro_n_thm}. As a consequence of Proposition
\ref{tate_is_generic} in Section \ref{Tate_main_result}, the Frobenius
twist of $J_l(\pi_F)$, defined as
$J_l(\pi_F)\otimes_{{\rm Frob}} \overline{\mathbb{F}}_l$, where
${\rm Frob}$ is the Frobenius automorphism of
$\overline{\mathbb{F}}_l$, is in fact the unique generic sub-quotient
of the Tate cohomology group $\widehat{H}^0(J_l(\pi_E))$. We use
Kirillov and Whittaker models of generic representations to prove our
main result.  The hypothesis that $l$ does not divide the pro-order of
${\rm GL}_{n-1}(F)$ is required in the proof of a vanishing result of
Rankin--Selberg integrals on ${\rm GL}_{n-1}(F)$ (the analogue of
\cite[Lemma 3.5]{MR620708} or \cite[6.2.1]{MR1981032}). This condition
on $l$ may be removed using $\gamma$-factors defined over local
Artinian $\overline{\mathbb{F}}_l$-algebras as defined in the work of
G.Moss and N.Matringe in \cite{matringe2022kirillov}. However, the
right notion of base change over local Artinian
$\overline{\mathbb{F}}_l$-algebras is not clear to the authors and
hence, we use the mild hypothesis that $l$ does not divide
$|{\rm GL}_{n-1}(\mathbb{F}_q)|$. If $\pi_F$ and $\pi_E$ are both
cuspidal, then using the Kirillov model for cuspidal representations,
one observes that the Tate cohomology group
$\widehat{H}^0(r_l(\pi_E))$ is an irreducible ${\rm GL}_n(F)$
representation, and the above theorem says that this Tate cohomology
space is isomorphic to the Frobenius twist of mod-$l$ reduction of
$\pi_F$ (Corollary \ref{n_thm}) when $l$ does not divide the pro-order
of ${\rm GL}_{n-1}(F)$. This is conjectured by Ronchetti in
\cite[Conjecture 2]{MR3551160}.
	
Let $\mathcal{K}$ be the maximal unramified extension of
$\mathbb{Q}_l$ in an algebraic closure $\overline{\mathbb{Q}}_l$ of
$\mathbb{Q}_l$.  Let $\Lambda$ be the ring of integers of
$\mathcal{K}$. We also prove an integral version of Theorem
\ref{intro_n_thm}.  To be precise, say $\pi_E$ is an integral generic
$\mathcal{K}$-representation of ${\rm GL}_n(E)$--which is absolutely
irreducible (i.e.,
$\pi_E\otimes_{\mathcal{K}} \overline{\mathbb{Q}}_l$ is irreducible),
such that $\pi_E\otimes_{\mathcal{K}} \overline{\mathbb{Q}}_l$ is the
base change lift of an $l$-adic integral generic representation
$\pi_F$ of ${\rm GL}_n(F)$. We show that the Frobenius twist of
$J_l(\pi_F)$ occurs as the unique generic subquotient of the zeroth
Tate cohomology group
$\widehat{H}^0(\mathbb{W}_\Lambda(\pi_E,\psi_E))$ (Corollary
\ref{noncuspidal_n_Whittaker_lattice_thm}), where
$\mathbb{W}_\Lambda(\pi_E,\psi_E)$ is the space of all
$\Lambda$-valued functions in the Whittaker model of $\pi_E$, with
respect to a ${\rm Gal}(E/F)$-equivariant character $\psi_E$. A
priori, Vign\'eras showed that $\mathbb{W}_\Lambda(\pi_E,\psi_E)$ is a
${\rm GL}_n(E)$-invariant $\Lambda$-lattice in
$\mathbb{W}(\pi_E, \psi_E)$.
	
When $l$ does not divide the pro-order of ${\rm GL}_n(F)$, we obtain a
much precise version of Theorem \ref{intro_n_thm}. We can show that the first
Tate cohomology group of any ${\rm Gal}(E/F)$ invariant lattice
$\mathcal{L}$ in a generic representation $\pi_E$ as in Theorem
\ref{intro_n_thm} is trivial. Moreover, we show that the zeroth Tate
cohomology group of the mod-$l$ reduction $r_l(\pi_E)$ is an
irreducible representation of ${\rm GL}_n(F)$ (see Theorem
\ref{Tate_L(D)} and Corollary \ref{Tate_0_generic}). Our method can
also be extended to some non-generic representations as
well. Especially for those irreducible representations of
${\rm GL}_n(E)$ which remains irreducible when restricted to the
mirabolic subgroup, denoted by $P_n(E)$. This class of representations
are exactly the Zelevinsky sub-representations. Assume that $\sigma_E$
is an $l$-adic cuspidal representation obtained as a base change
lifting of $\sigma_F$ to ${\rm GL}_n(E)$.  Let $\Delta$ be a segment
(see Section \ref{segment}) on the cuspidal line of $\pi_E$ (defined
in Theorem \ref{intro_n_thm}). We apply Theorem \ref{intro_n_thm} to
compute the Tate cohomology of mod-$l$ Zelevinsky sub-representations
$Z(\overline{\Delta})$ (see Theorem \ref{Tate_Z(D)}), where
$\overline{\Delta}$ is the segment on the cuspidal line of
$r_l(\sigma_E)$.
	
When $F$ is a local function field, the above theorem follows
from the work of T.Feng \cite{feng2024smith}. T.Feng
uses the constructions of V. Lafforgue and A. Genestier-V.Lafforgue
\cite{genestier2017chtoucas}. Assuming that $l$ and $p$ do
not divide $n$, Ronchetti proved the above result for
depth-zero cuspidal representations using the compact
induction model. Our methods are very different from the work
of N.Ronchetti and the work of T.Feng. We rely on
Rankin--Selberg integrals and Whittaker models. We
do not require the explicit construction of cuspidal
representations. We use various properties of local $\epsilon$
and $\gamma$-factors both in $l$-adic and mod-$l$ situations
associated with the representations of the $p$-adic group and
the Weil group. The machinery of local $\epsilon$ and
$\gamma$-factors of both $l$-adic and mod-$l$ representations
of ${\rm GL}_n(F)$ is made available by the seminal works of
D.Helm, G.Moss, N.Matringe and R.Kurinczuk (see
\cite{MR3867634}, \cite{MR3556429}, \cite{MR4311563},
\cite{MR3595906}).
	
The case where $\pi_E$ is a cuspidal representation of ${\rm GL}_2(E)$
is considered in Theorem \ref{n=2_thm}, the general case is proved in
Theorem \ref{noncuspidal_n_thm} using induction on $n$. The reader
might quickly follow the proof of Theorem \ref{n=2_thm} before going
to the general case.  We sketch the proof of Theorem
\ref{intro_n_thm}. The theorem is proved, inductively on $n$, using
the Kirillov model. Let
$\psi_F:F \rightarrow \overline{\mathbb{Q}}_l^\times$ be a non-trivial
additive character and let $\psi_E$ be the character
$\psi_F\circ{\rm Tr}_{E/F}$, where ${\rm Tr}_{E/F}:E\rightarrow F$ is
the trace function. Let $(\pi_F, V)$ be an integral generic $l$-adic
representation of ${\rm GL}_n(F)$. In particular, $V$ is a
$\overline{\mathbb{Q}}_l$-vector space. Let $N_n(F)$ be the group of
unipotent upper triangular matrices in ${\rm GL}_n(F)$. Let
$\Theta_F : N_n(F)\rightarrow \overline{\mathbb{Q}}_l^\times$ be a
non-degenerate character corresponding to $\psi_F$. We denote by
$\mathbb{W}(J_l(\pi_F),\overline{\psi}_F)$ the Whittaker model of the
unique generic sub--quotient, denoted by $J_l(\pi_F)$, of the mod-$l$
reduction of $\pi_F$. Here, $\overline{\psi}_F$ is the mod-$l$
reduction of $\psi_F$. Let $\pi_E$ be the base change lift of $\pi_F$
to ${\rm GL}_n(E)$. Similar notations for $\pi_E$ are followed where
$\overline{\psi}_F$ is replaced with $\overline{\psi}_E$. It is easy
to note that (Lemma \ref{inv_whit_model})
$\mathbb{W}(J_l(\pi_E),\overline{\psi}_E)$ is stable under the action
of ${\rm Gal}(E/F)$ on the space
${\rm Ind}_{N_n(E)}^{{\rm GL}_n(E)}\overline{\Theta}_E$. Let
$\mathbb{K}(J_l(\pi_F),\overline{\psi}_F)$ be the Kirillov model of
$J_l(\pi_F)$. Using the result \cite[Theorem
4.2]{matringe2022kirillov}, we get that the restriction to $P_n(F)$
map from $\mathbb{W}(J_l(\pi_F), \overline{\psi}_F)$ to
	$\mathbb{K}(J_l(\pi_F), \overline{\psi}_F)$ is a bijection.
	
        The Kirillov model $\mathbb{K}(J_l(\pi_E), \overline{\psi}_E)$
        contains the space of all smooth and compactly supported
        functions $\ind_{N_n(E)}^{P_n(E)}\overline{\Theta}_E$.  Recall
        that the Tate cohomology group
        $\widehat{H}^0(\mathbb{K}(J_l(\pi_E), \overline{\psi}_E))$
        (For definition, see Section \ref{Tate_cohomology}) is an
        $\overline{\mathbb{F}}_l$-representation of $P_n(F)$. Let
        $\Phi$ be the following map obtained by restriction of
        functions to $P_n(F)$:
$$ \Phi: \widehat{H}^0(\mathbb{K}(J_l(\pi_E), \overline{\psi}_E))
\longrightarrow {\rm Ind}_{N_n(F)}^{P_n(F)}\overline{\Theta}_F^{l}. $$ 
Using compactly supported functions, one can show that the inverse
image of $\mathbb{K}(J_l(\pi_F)^{(l)}, \overline{\psi}_F^l)$
under the map $\Phi$ is non-zero, and it is denoted by
$\mathcal{M}(\pi_F, \psi_F)$. Here, $J_l(\pi_F)^{(l)}$ is
the Frobenius twist of $J_l(\pi_F)$. To prove the main theorem, 
we show that
the space $\mathcal{M}(\pi_F, \psi_F)$ is stable under
${\rm GL}_n(F)$ and the restriction of $\Phi$ to the space
$\mathcal{M}(\pi_F, \psi_F)$ is ${\rm GL}_n(F)$
equivariant. This is just equivalent to showing that
\begin{equation}\label{intro_eq_rk}
	I\big(X, \Phi(J_l(\pi_E)(w_n)W),\sigma(w_{n-1})W'\big) =
	I\big(X,J_l(\pi_F)^{(l)}(w_n)\Phi(W),\sigma(w_{n-1})W'\big),
\end{equation}
for all $W\in \mathcal{M}(\pi_F, \psi_F)$ and
$W'\in \mathbb{W}(\sigma, \overline{\psi}_F^{-l})$, where $w_{n-1}, w_n$
are defined in subsection (\ref{15}); and $\sigma$ is an $l$-modular
generic representation of ${\rm GL}_{n-1}(F)$. Here, $I(X, W, W')$ is
a mod-$l$ Rankin--Selberg zeta functions written as a formal power
series in the variable $X$ instead of the traditional $q^{-s}$
(\cite[Section 3]{MR3595906}). We transfer the local Rankin--Selberg
zeta functions $I\big(X, \Phi(J_l(\pi_E)(w_n)W),\sigma(w_{n-1})W'\big)$
made from integrals on ${\rm GL}_{n-1}(F)/N_{n-1}(F)$ to
Rankin--Selberg zeta functions defined by integrals on
${\rm GL}_{n-1}(E)/N_{n-1}(E)$. Then, using local Rankin--Selberg
functional equation, we show that the equality in (\ref{intro_eq_rk})
is equivalent to certain identities of mod-$l$ local $\gamma$-factors,
such as \eqref{most_important_identity}.
	
We briefly explain the contents of this article. In Section \ref{x},
we recall various notations, conventions on integral representations,
Whittaker models and Kirillov models. In Section
\ref{Deligne_rep_section}, we collect various results on local
constants both in mod-$l$ and $l$-adic settings. In Section
\ref{LLC_section}, we put some well known results from $l$-adic local
Langlands correspondence. In Section \ref{Tate_cohomology}, we recall
and set up some initial results on Tate cohomology of smooth integral
representations as well as mod-$l$ representations. In Section
\ref{Tate_main_result}, we begin with a few observations on
compatibility of Jacquet and twisted Jacquet functors with Tate
cohomology. Then we prove our main result Theorem
\ref{noncuspidal_n_thm}. In Sections \ref{Tate_Zelevinsky} and
\ref{Tate_generic}, in the banal case, we completely compute the Tate
cohomology of the representations $Z(\Delta)$ and $L(\Delta)$ using
Theorem \ref{noncuspidal_n_thm}.
	
{\it Acknowledgements} We express our deep gratitude to the anonymous
referee for his careful reading of the manuscript and for various
suggestions. We thank Tony Feng for his interest and encouragement,
and to Akshay Venkatesh and Guy Henniart for their interest in this
work.
	
\section{Preliminaries}\label{x}
\subsection{}\label{16}
Let $K$ be a non-Archimedean local field and let
$\mathfrak{o}_K$ be the ring of integers of $K$. Let
$\mathfrak{p}_K$ be the maximal ideal of $\mathfrak{o}_K$ and
let $\varpi_K$ be a uniformizer of $K$. Let $q_K$ be the
cardinality of the residue field
$k_K=\mathfrak{o}_K/\mathfrak{p}_K$. Let
$\upsilon_K: K^\times \rightarrow \mathbb{Z}$ be the
normalized valuation. We denote by $\nu_K$ the normalized
absolute value of $K$ corresponding to $\upsilon_K$. Let $l$
and $p$ be two distinct odd primes. Let $F$ be a finite
extension of $\mathbb{Q}_p$, and let $E$ be a finite Galois
extension of $F$ with $[E:F] = l$. We denote the group
${\rm Gal}(E/F)$ by $\Gamma$.
	
\subsection{}\label{15}
For any ring $A$, let $M_{r \times s}(A)$ be the $A$-algebra
of all $r \times s$ matrices with entries from $A$. Let
$GL_n(K) \subseteq M_{n \times n}(K)$ be the group of all
invertible $n \times n$ matrices. We denote by $G_n(K)$ the
group ${\rm GL}_n(K)$ and $G_n(K)$ is equipped with locally
compact topology induced from the local field $K$. For
$r \in \mathbb{Z}$, let
$$ G_n^r(K) = \{g \in G_n(K) : \upsilon_K(\det(g)) = r\}. $$ 
We set $P_n(K)$, the mirabolic subgroup, defined as the group:
\begin{center}
$\bigg\{
\begin{pmatrix}
	A  &  M\\
	0  &  1
\end{pmatrix}
: A \in G_{n-1}(K), M \in M_{(n-1) \times 1}(K)\bigg\}$.
\end{center}
Let $B_n(K)$ be the group of all invertible upper triangular matrices
in $M_{n\times n}(K)$, and let $N_n(K)$ be its unipotent radical. We
denote by $w_n$ the following matrix of $G_n(K)$ :
\begin{center} $w_n =
\begin{pmatrix}
	0 &  &  &  & 1 \\
	&  &  &  1  \\
	&  &  . \\
	&  .   \\
	1 &  &  &  &  0
\end{pmatrix}.$ 
\end{center}
Let $X_K$ denote the coset space $N_{n-1}(K) \setminus
G_{n-1}(K)$. For $r \in \mathbb{Z}$, we denote the coset space
$\{N_{n-1}(K)g: g\in  G_{n-1}^r(K)\}$ by $X_K^r$.
	
\subsection{}
Fix an algebraic closure $\mathbb{\overline{Q}}_l$ of the field
$\mathbb{Q}_l$.  Let $\mathbb{\overline{Z}}_l$ be the integral closure
of $\mathbb{Z}_l$ in $\mathbb{\overline{Q}}_l$ and let
$\mathfrak{P}_l$ be the unique maximal ideal of
$\mathbb{\overline{Z}}_l$. We have
$\mathbb{\overline{Z}}_l/\mathfrak{P}_l \simeq
\mathbb{\overline{F}}_l$. We fix a square root of $q_F$ in
$\overline{\mathbb{Q}}_l$, and it is denoted by $q_F^{1/2}$. The
choice of $q_F^{1/2}$ is required for transferring the complex local
Langlands correspondence to a local $l$-adic Langlands correspondence
(see \cite[Chapter 8]{MR2234120}). Let $\mathcal{K}$ denote the
maximal unramified extension of $\mathbb{Q}_l$ in
$\overline{\mathbb{Q}}_l$, and let $\Lambda$ be the ring of integers
of $\mathcal{K}$. Let $W(\overline{\mathbb{F}}_l)$ be the ring of Witt
vectors of $\overline{\mathbb{F}}_l$. The prime number $l$ is called
{\it banal} for $G_n(K)$ if $l$ does not divide $|{\rm GL}_n(k_K)|$.
	
\subsection{Smooth representations and Integral representations}
Let $G$ be a locally compact and totally disconnected group. A
representation $(\pi,V)$ is said to be smooth if, for every vector
$v \in V$, the $G$-stabilizer of $v$ is an open subgroup of $G$. All
the representations are assumed to be smooth and the representation
spaces are vector spaces over $R$, where $R = \mathbb{\overline{Q}}_l$
or $\mathbb{\overline{F}}_l$. A representation $(\pi,V)$ is called
{\it $l$}-adic when $R=\mathbb{\overline{Q}}_l$ and $(\pi,V)$ is
called $l$-modular when $R=\mathbb{\overline{F}}_l$. We denote by
${\rm Irr}(G,R)$, the set of all irreducible smooth
$R$-representations of $G$. Let $C_c^\infty(G,R)$ denote the set of
all locally constant and compactly supported functions on $G$ taking
values in a ring $R$.
	
Let $(\pi,V)$ be an $l$-adic representation of $G$. A lattice in $V$
is a free $\mathbb{\overline{Z}}_l$-module $\mathcal{L}$ such that
$\mathcal{L} \otimes_{\mathbb{\overline{Z}}_l} \mathbb{\overline{Q}}_l
= V$. The representation $(\pi, V)$ is said to be integral if it has
finite length as a representation of $G$ and there exists a
$G$-invariant lattice $\mathcal{L}$ in $V$. A character is a smooth
one-dimensional representation $\chi: G \longrightarrow
R^\times$. For $G=G_n(K)$, a character
$\chi :K^\times \rightarrow R^\times$ induces a character
$\chi \circ {\rm det} : G_n(K) \rightarrow R^\times$. By abuse of
notation, we denote the character $\chi \circ {\rm det}$ by $\chi$. In
particular, the normalized absolute value of $K$ gives a character
$\nu_K$ of $G_n(K)$.  We say that a character
$\chi : G \longrightarrow \mathbb{\overline{Q}}_l^\times$ is integral
if it takes values in $\mathbb{\overline{Z}}_l$.
	
Let $(\pi,V)$ be an integral $l$-adic representation of $G$. Choose a
$G$-invariant lattice $\mathcal{L}$ in $V$. Then the group $G$ acts on
$\mathcal{L}
\otimes_{\mathbb{\overline{Z}}_l}\overline{\mathbb{F}}_l$, which is a
vector space over $\mathbb{\overline{F}}_l$. This gives an $l$-modular
representation, which depends on the choice of the $G$-invariant
lattice $\mathcal{L}$. By \cite[II. 5.11.a and 5.11.b]{MR1395151}, the representation
$\big(\pi, \mathcal{L}
\otimes_{\mathbb{\overline{Z}}_l}\overline{\mathbb{F}}_l\big)$ is of
finite length and its semisimplification is independent of the choice
of the $G$-invariant lattice $\mathcal{L}$ in $V$. We denote by
$r_l(\pi)$ the semisimplification of
$\big(\pi, \mathcal{L}
\otimes_{\mathbb{\overline{Z}}_l}\overline{\mathbb{F}}_l\big)$. The
representation $r_l(\pi)$ is called the mod-$l$ reduction of the
$l$-adic representation $\pi$. We say that an $l$-modular
representation $\sigma$ lifts to an integral $l$-adic representation
$\pi$ if there exists a $G$-invariant lattice
$\mathcal{L} \subseteq \pi$ such that
$\mathcal{L} \otimes_{\mathbb{\overline{Z}}_l}\overline{\mathbb{F}}_l
\simeq \sigma$.
	
\subsection{Parabolic induction.}\label{ss}
Let $H$ be a closed subgroup of $G$. Let ${\rm Ind}_H^G$ and ${\rm
ind}_H^G$ be the smooth induction functor and compact induction
functor respectively. We follow \cite{MR579172} for the definitions.
	
Set $G = G_n(K)$, $P = P_n(K)$ and $N = N_n(K)$, where $G_n(K)$,
$P_n(K)$ and $N_n(K)$ are defined in subsection (\ref{15}). Let
$\lambda =(n_1,n_2,....,n_t)$ be an ordered partition of $n$. Let
$Q_\lambda\subseteq G_n(K)$ be the group of matrices of the form
\begin{center}
$\begin{pmatrix}
	A_1 & * & * & * & *\\
	& A_2 & * & * & *\\
	&   & . & * & * \\
	&   &   & . & *\\
	&   &   &   & A_t
\end{pmatrix}$,
\end{center}
where $A_i \in G_{n_i}(K)$, for all $1\leq i\leq t$. Then
$Q_\lambda = M_\lambda \ltimes U_\lambda$, where $M_\lambda$ is the
group of block diagonal matrices of the form
\begin{center}
$\begin{pmatrix}
	A_1 &   &   &   &  \\
	& A_2 &   &   &  \\
	&   & . &   &   \\
	&   &   & . &  \\
	&   &   &   & A_t 
\end{pmatrix}$, $A_i \in G_{n_i}(K)$,
\end{center}
for all $1\leq i\leq t$
and $U_\lambda$ is the unipotent radical of $Q_\lambda$ consisting of
matrices of the form
\begin{center}
$U_\lambda =
\begin{pmatrix}
	I_{n_1} & * & * & * & *\\
	& I_{n_2} & * & * & *\\
	&   & . & * & * \\
	&   &   & . & *\\
	&   &   &   & I_{n_t} 
\end{pmatrix}$, 
\end{center}
where $I_{n_i}$ is the ${n_i} \times {n_i}$ identity matrix.
	
Let $\sigma$ be an $R$-representation of $M_\lambda$. Then the
representation $\sigma$ is considered as a representation of
$Q_\lambda$ by inflation via the map
$Q_\lambda \rightarrow Q_\lambda/U_\lambda \simeq M_\lambda$. The
induced representation ${\rm
Ind}_{Q_\lambda}^G(\sigma)$ is called the parabolic induction of
$\sigma$. We denote the normalized parabolic induction of
$\sigma$ corresponding to the partition
$\lambda$ by
$i_{Q_\lambda}^G(\sigma)$. For details, see \cite{MR579172}. Let
$\lambda = (n_1,....,n_s)$ be a partition of $n$ and let
$\sigma_i$ be $R$-representation of $G_{n_i}$ for each
$i$. We denote the parabolic induction $i_{Q_\lambda}^G(\sigma_1
\otimes\cdots\otimes \sigma_s)$ by the product symbol $\sigma_1
\times\cdots\times \sigma_s$.
	
\subsubsection{}
Let $\lambda$ be an ordered partition of $n$. Let
$\sigma$ be an integral $l$-adic representation of
$M_\lambda$ and let $\mathcal{L}$ be a
$G$-invariant lattice in
$\sigma$. Then by \cite[I. 9.3]{MR1395151}, the space
$i_{Q_\lambda}^G(\mathcal{L})$, consisting of functions in
$i_{Q_\lambda}^G(\sigma)$ taking values in
$\mathcal{L}$, is a
$G$-invariant lattice in $i_{Q_\lambda}^G(\sigma)$. Moreover, we have
$$ i_{Q_\lambda}^G(\mathcal{L} \otimes_{\mathbb{\overline{Z}}_l}
\mathbb{\overline{F}}_l) \simeq i_{Q_\lambda}^G(\mathcal{L})
\otimes_{\mathbb{\overline{Z}}_l} \mathbb{\overline{F}}_l. $$
Hence parabolic induction commutes with reduction modulo $l$ that is,
$$ 
r_l(i_{Q_\lambda}^G(\sigma)) \simeq [i_{Q_{\lambda}}^G(r_l(\sigma))], 
$$
where the square bracket denotes the semisimplification of
$i_{Q_\lambda}^G(r_l(\sigma))$. 
	
\subsection{Cuspidal and Supercuspidal representation}
Keeping the notation as in (\ref{ss}). Let $\pi$ be an irreducible
$R$-representation of $G$. Then $\pi$ is called a cuspidal
representation if for all proper subgroups
$Q_\lambda =M_\lambda \ltimes U_\lambda$ of $G$ and for all
irreducible $R$-representations $\sigma$ of $M_\lambda$, we have
\begin{center}
$\rm Hom$$_G(\pi, i_Q^G(\sigma)) = 0$.
\end{center}
The representation
$\pi$ is called supercuspidal if for all proper subgroups $Q_\lambda =
M_\lambda \ltimes U_\lambda$ of
$G$ and for all irreducible $R$-representations
$\sigma$ of $M_\lambda$, the representation
$(\pi,V)$ is not a subquotient of $i_Q^G(\sigma)$.
\begin{remark}
\rm Let $k$ be an algebraically closed field, and let
$\pi$ be a smooth $k$-representation of $G$. If the characteristic of
$k$ is $0$, then $\pi$ is cuspidal if and only if
$\pi$ is supercuspidal. But when characteristic of $k$ is $l >
0$, there are cuspidal representations of
$G$ which are not supercuspidal. For details, see \cite[Section 2.5,
Chapter 2]{MR1395151}.
\end{remark}
	
\subsection{Generic representation}\label{rr}
Let $\psi_K : K \rightarrow
R^\times$ be a non-trivial additive character of $K$. Let
$\Theta_K$ be the character of $N_n(K)$, defined by
$$ \Theta_K(x_{ij}) := \psi_K(\sum_{i=1}^{n-1}x_{i,i+1}). $$
Let $(\pi,V)$ be an irreducible $R$-representation of
$G_n(K)$. Then recall that
$$ {\rm dim}_R\big({\rm Hom}_{N_n(K)}(\pi,\Theta_K)\big) \leq 1. $$
For the proof, see \cite{MR0425030} when
$R=\mathbb{\overline{Q}}_l$ and see \cite{MR1395151} when
$R=\mathbb{\overline{F}}_l$.  An irreducible
$R$-representation $(\pi,V)$ of $G_n(K)$ is called {\it generic} if
$$ {\rm dim}_R\big({\rm Hom}_{N_n(K)}(\pi,\Theta_K)\big) = 1. $$
	
\subsubsection{Whittaker Model}\label{whittaker_recap}
Let $(\pi,V)$ be a generic $R$-representation of $G_n(K)$. By
Frobenius reciprocity, the representation $\pi$ is embedded in the
space ${\rm Ind}_{N_n(K)}^{G_n(K)}(\Theta_K)$. Let $\mathcal{W}_{\pi}$
be a non-zero linear functional in the space
${\rm Hom}_{N_n(K)}(\pi,\Theta_K)$. Let
$\mathbb{W}(\pi,\psi_K) \subset {\rm Ind}_{N_n(K)}^{G_n(K)}(\Theta_K)$
be the space consisting of functions $W_v$, $v \in V$, where
$$ W_v(g) := \mathcal{W}_{\pi}\big(\pi(g)v\big), $$
for all $g \in G_n(K)$. Then the map $v \mapsto W_v$ induces an isomorphism
from $(\pi,V)$ to $\mathbb{W}(\pi,\psi_K)$. 
	
\subsubsection{Segments}\label{segment}
In this subsection, we recall the notion of segments and its
associated representations. For details, see \cite{MR584084} for
$R = \overline{\mathbb{Q}}_l$ and \cite{MR3595906}, \cite{MR3178433}
for $R = \overline{\mathbb{F}}_l$.
	
Let $r,t \in \mathbb{Z}$ with $r \leq t$. A segment is a sequence
$\Delta = (\nu_K^r\sigma, \nu_K^{r+1}\sigma,...,\nu_K^t\sigma)$, with
$\sigma$ a cuspidal $R$-representation of $G_n(K)$. The length of
$\Delta$ is defined to be $t-r+1$. In \cite[Definition 7.5]{MR3178433}
the authors, using Bushnell-Kutzko's simple types and the Hecke algebras
associated with them, defined a certain quotient of the
parabolically induced representation
$$ \tau = \nu_K^r\sigma \times \nu_K^{r+1}\sigma
\times\cdots\times \nu_K^t\sigma,$$ denoted by $\mathcal{L}(\Delta)$. The
normalised Jacquet module of $\mathcal{L}(\Delta)$ with respect to the
opposite of the parabolic subgroup $P_{(n,\cdots, n)}$ is equal to
$$ \nu_K^r\sigma \otimes \nu_K^{r+1}\sigma
\otimes\cdots \otimes \nu_K^t\sigma. $$
Moreover, there is a unique generic sub-quotient of $\tau$, denoted by
${\rm St}(\sigma,[r,t])$ and it is called the generalised Steinberg
representation associated to $\Delta$. We denote by
${\rm St}(\sigma,k)$ the representation ${\rm St}(\sigma,[0,k-1])$,
for $k\geq 1$.
	
\subsubsection{}
Let $\sigma$ be a cuspidal $R$-representation of $G_n(K)$. The set
$\{\nu_K^r\sigma :r \in \mathbb{Z}\}$ is called the {\it cuspidal
line} of $\sigma$ and the cardinality of this set is denoted by
$o(\sigma)$. Recall that \cite[Section 5.2]{MR3178433} defines a
positive integer $e(\sigma)$ as follows:
\begin{equation}
e(\sigma)= 
\begin{cases}
		+\infty   & \text{if } R=\overline{\mathbb{Q}}_l;\\
		o(\sigma) & \text{if } R=\overline{\mathbb{F}}_l \
		\text{and } o(\sigma) > 1;\\
		l & \text{if } R=\overline{\mathbb{F}}_l\ \text{and
		} o(\sigma) =1.
\end{cases}
\end{equation}
Then for a segment $\Delta = (\nu_K^r\sigma,\dots,\nu_K^t\sigma)$,
with $r \leq t$, the representation $\mathcal{L}(\Delta)$ is equal to
${\rm St}(\sigma,[r,t])$ if and only if the length of the segment
$\Delta$ is less than $e(\sigma)$(\cite[Remarque 8.14]{MR3178433}). In
this case, the segment $\Delta$ is called a generic segment. Note that
every segment is generic for $R=\overline{\mathbb{Q}}_l$.
	
\subsubsection{}
Two segments $\Delta_1$ and $\Delta_2$ are said to be linked if
$\Delta_1 \nsubseteq \Delta_2$, $\Delta_2 \nsubseteq \Delta_1$ and
$\Delta_1 \cup \Delta_2$ is a segment. The following theorem is proved
by \cite[Theorem 9.10]{MR3178433} for $R = \overline{\mathbb{F}}_l$
and \cite[Theorem 9.7]{MR584084} for $R = \overline{\mathbb{Q}}_l$.
\begin{theorem}
Let
$\pi=\mathcal{L}(\Delta_1) \times\cdots\times
\mathcal{L}(\Delta_t)$ be an $R$-representation of $G_n(K)$, where
each $\Delta_j$ is a generic segment. Then $\pi$ is irreducible if and
only if the segments $\Delta_i$ and $\Delta_j$ are not linked for
all $i,j$ with $i \ne j$.
\end{theorem}  
An $R$-representation of the form
$\mathcal{L}(\Delta_1)\times\cdots\times \mathcal{L}(\Delta_t)$, where
each $\Delta_i$ is generic, is called a representation of {\it
  Whittaker type}. In \cite[Theorem 9.7]{MR584084} and
\cite[Proposition V.3]{Vigneras_Induced}, it is shown that
\begin{theorem}
An $R$-representation $\pi$ of $G_n(K)$ is generic if and only
if $\pi$ is an irreducible $R$-representation of Whittaker type.
\end{theorem}
	
\subsubsection{}\label{SL}
In this subsection, we fix a standard lift of an $l$-modular generic
representation of $G_n(K)$. First, recall that any $l$-modular cuspidal
representation of $G_m(K)$ can be lifted to an $l$-adic cuspidal
representation of $G_m(K)$ (see \cite[Chapter 3,
4.25]{MR1395151}). Let $\rho$ be an $l$-modular cuspidal
representation of $G_m(K)$ and let
$\Delta = (\rho,\overline{\nu}_K\rho,\dots,
\overline{\nu}_K^{r-1}\rho)$ be a segment associated with $\rho$,
where $\overline{\nu}_K$ is the mod-$l$ reduction of $\nu_K$. Let
$\sigma$ be a cuspidal lift of $\rho$. Then the segment
$D = (\sigma,\nu_K\sigma,\dots,\nu_K^{r-1}\sigma)$ is called a
standard lift of $\Delta$. When
$\mathcal{L}(\Delta) = {\rm St}(\rho,r)$, then the mod-$l$
representation $\mathcal{L}(\Delta)$ lifts to
$\mathcal{L}(D) = {\rm St}(\sigma,r)$ (\cite[Proposition
2.16]{MR3595906}).
	
Let $\pi$ be a generic $l$-modular representation of $G_n(K)$. Then
$\pi$ is of the form
$\mathcal{L}(\Delta_1)\times\cdots\times \mathcal{L}(\Delta_t)$, where
each $\Delta_i$ is a generic segment. For each $1\leq i \leq t$, let
$D_i$ be a standard lift of $\Delta_i$. Then the $l$-adic
representation
$\tau = \mathcal{L}(D_1)\times\cdots\times \mathcal{L}(D_t)$ is
generic, and $\pi$ lifts to $\tau$ (\cite[Remark
2.31]{MR3595906}). The representation $\tau$ is called a {\it standard
lift} of $\pi$.
	
\subsubsection{}\label{vigneras_whittaker_lattice}
Let $(\pi,V)$ be an integral, $l$-adic, generic representation of
$G_n(K)$. We fix a non-trivial additive character
$\psi_K:K\rightarrow \Lambda^\times$. By abuse of notation, the
composition
$K\xrightarrow{\psi_K} \Lambda^\times \hookrightarrow
\overline{\mathbb{Q}}_l^\times$ is also denoted by $\psi_K$. Consider
the space $\mathbb{W}^0(\pi,\psi_K)$ consisting of
$W \in \mathbb{W}(\pi,\psi_K)$, taking values in
$\mathbb{\overline{Z}}_l$. It follows from \cite[Theorem 2]{MR2058628}
that the $\overline{\mathbb{Z}}_l$-module $\mathbb{W}^0(\pi,\psi_K)$
is a $G_n(K)$-invariant lattice in $\mathbb{W}(\pi,\psi_K)$. The
lattice $\mathbb{W}^0(\pi,\psi_K)$ is also called the {\it integral
  Whittaker model} or {\it Whittaker lattice}. Let $\tau$ be an
$l$-modular generic representation of $G_n(K)$, and let $\pi$ be an
$l$-adic generic representation of $G_n(K)$. Then, the representation
$\pi$ is called a Whittaker lift of $\tau$ if there exists a lattice
$\mathcal{L} \subseteq \mathbb{W}^0(\pi,\psi_K)$ such that
$$ \mathcal{L} \otimes_{\overline{\mathbb{Z}}_l} \overline{\mathbb{F}}_l \simeq
\mathbb{W}(\tau, \overline{\psi}_K), $$
where $\overline{\psi}_K$ is the reduction mod-$l$ of $\psi_K$. Note
that any standard lift of an $l$-modular generic representation $\pi$
is a Whittaker lift (see \cite[Theorem 2.26]{MR3595906}). Let 
$(\pi, V)$ be an $l$-adic integral generic representation defined 
over the field $\mathcal{K}$. Let $V'$ be a $\mathcal{K}$-structure 
in $V$. There exists a $\Lambda[G_n(K)]$-lattice in $V'$ 
(see \cite[Corollary 4.4.4]{MR3250061}). Thus, we get that 
the set of $\Lambda$-valued functions, denoted by 
$\mathbb{W}_\Lambda(\pi, \psi_K)$, in $\mathbb{W}^0(\pi, \psi_K)$
is a $G_n(K)$-stable lattice. 
	
\subsubsection{}\label{sr}
Now we follow the notations as in (\ref{16}). Choose a generator
$\gamma$ of $\Gamma$. Let $\pi$ be an $R$-representation of
$G_n(E)$. The group $\Gamma = {\rm Gal}(E/F)$ acts on
$G_n(E)$ component-wise i.e., for $\gamma \in \Gamma, g =
(a_{ij})_{i,j=1}^n \in G_n(E)$, we set
\begin{center}
$\gamma. g := (\gamma(a_{ij}))_{i,j =1}^n$.
\end{center}
Let $\pi^\gamma$ be the representation of $G_n(E)$ on $V$, defined by
\begin{center}
$\pi^\gamma (g) := \pi(\gamma.g)$, for all $g \in G_n(E)$.
\end{center}
We say that the representation $\pi$ of
$G_n(E)$ is $\Gamma$-equivariant if the representations
$\pi$ and
$\pi^\gamma$ are isomorphic.  We now prove a lemma concerning the
$\Gamma$ invariance of the Whittaker model of a $\Gamma$-equivariant
representation $\pi$ of $G_n(E)$. Let $\psi_F$ and
$\psi_E$ be the non-trivial additive characters of $F$ and 
$E$ respectively such that $\psi_E = \psi_F \circ {\rm
Tr}_{E/F}$ where, ${\rm
Tr}_{E/F}$ is the trace map of the extension $E/F$. Let
$\Theta_F$ and $\Theta_E$ be the characters of $N_n(F)$ and
$N_n(E)$ respectively, as defined in (\ref{rr}).
Now consider the action of $\Gamma$ on the space ${\rm
Ind}_{N_n(E)}^{G_n(E)}(\Theta_E)$, given by
$$ (\gamma . f)(g) := f(\gamma^{-1}g), $$
for all $\gamma \in \Gamma$, $g \in G_n(E)$ and
$f \in \rm Ind$$_{N_n(E)}^{G_n(E)}(\Theta_E)$.
\begin{lemma}\label{inv_whit_model}
Let $(\pi, V)$ be a generic $R$-representation of $G_n(E)$ such 
that $(\pi, V)$ is $\Gamma$-equivariant. Then the Whittaker model
$\mathbb{W}(\pi, \psi_E)$ of $\pi$ is invariant under the action
of $\Gamma$.
\end{lemma}
\begin{proof}
  Let
  $\mathcal{W}_{\pi}$ be a Whittaker functional on the representation
  $\pi$. For $v \in V$, we have
$$
\mathcal{W}_{\pi}(\pi^\gamma(x)v)= \Theta_E(\gamma(x))
\mathcal{W}_{\pi}(v) = (\psi_F
\circ {\rm Tr}_{E/F})\big(\sum_{i=1}^{n-1}\gamma
(x_{i,i+1})\big)\mathcal{W}_{\pi}(v) = 
\Theta_E(x)\mathcal{W}_{\pi}(v),
$$ 
for all $x \in N_n(E)$. Thus,
$\mathcal{W}_{\pi}$ is also a Whittaker functional for 
the representation $(\pi^\gamma, V)$. Let
$W_v \in \mathbb{W}(\pi,\psi_E)$. Then
$$ (\gamma^{-1}.W_v)(g) = \mathcal{W}_{\pi}(\pi^\gamma(g)v). $$ 
From the uniqueness of the Whittaker model, we have
$\gamma^{-1}.W_v \in
\mathbb{W}(\pi,\psi_E)$. Hence the lemma.
\end{proof}
\subsection{Kirillov Model}\label{kirrilov_model}
Let $\pi$ be a generic $R$-representation of
$G_n(K)$. Following the notations as in the subsections (\ref{ss}) and
(\ref{rr}), consider the space $\mathbb{K}(\pi,
\psi_K)$ of all elements $W$ restricted to $P_n(K)$, where
$W$ varies over $\mathbb{W}(\pi,
\psi_K)$. Then the space
$\mathbb{K}(\pi,\psi_K)$ is
$P_n(K)$-invariant. By Frobenius reciprocity, there is a non-zero
(unique upto a scalar) linear map $A_\pi : V \longrightarrow {\rm
  Ind}_{N_n(K)}^{P_n(K)}(\Theta_K)$, which is injective and compatible
with the action of $P_n(K)$. In fact,
\begin{center}
$A_\pi(V) = \mathbb{K}(\pi, \psi_K) \simeq \mathbb{W}(\pi,
\psi_K) \simeq \pi$.
\end{center}
Moreover, $\mathcal{K}(\psi_K) = {\rm ind}_{N_n(K)}^{P_n(K)}
(\Theta_K) \subseteq \mathbb{K}(\pi,
\psi_K)$ and the equality holds if
$\pi$ is cuspidal. The space of all
$\mathbb{\overline{Z}}_l$-valued functions in
$\mathbb{K}(\pi,\psi_K)$ (resp.
$\mathcal{K}(\psi_K)$) is denoted by $\mathbb{K}^0(\pi,
\psi_K)$ (resp. $\mathcal{K}^0(\psi_K)$).
	
\subsubsection{}
We now recall the Kirillov model for
$n=2$ and some of its properties. For details, see \cite{MR2234120}.
Up to isomorphism, any irreducible representation of
$P_2(K)$, which is not a character, is isomorphic to
\begin{equation}\label{f}
J_\psi := {\rm ind}_{N_2(K)}^{P_2(K)}(\psi),
\end{equation}
for some non-trivial smooth additive character $\psi$ of
$K$, viewed as character of $N_2(K)$ via standard isomorphism $N_2(K)
\simeq K$. Two different non-trivial characters of
$N_2(K)$ induce isomorphic representations of
$P_2(K)$. The space (\ref{f}) is identified with the space of locally
constant compactly supported functions on
$K^\times$, to be denoted by $C_c^\infty(K^\times,
\mathbb{\overline{Q}}_l)$. The action of
$P_2(K)$ on the space $C_c^\infty(K^\times,
\mathbb{\overline{Q}}_l)$ is given by
\begin{center}
$\bigg[J_\psi
\begin{pmatrix}
	a & 0\\
	0 & 1
\end{pmatrix}
f\bigg](y) = f(ay)$,\\
$\bigg[J_\psi
\begin{pmatrix}
	1 & x\\
	0 & 1
\end{pmatrix}
f\bigg](y) = \psi(xy) f(y)$,
\end{center}
for $a, y \in K^\times$ and $x \in K$.  For any cuspidal
representation of $(\pi, V)$ of $G_2(K)$, we get a model for the
representation $(\pi, V)$ on the space
$C_c^\infty(K^\times,\mathbb{\overline{Q}}_l)$. The action of
the group $G_2(K)$ on
$C_c^\infty(K^\times,\mathbb{\overline{Q}}_l)$ is denoted by
$\mathbb{K}_\psi^\pi$; by definition the restriction of
$\mathbb{K}_\psi^\pi$ to $P_2(K)$ is isomorphic to $J_\psi$. The
operator $\mathbb{K}_\psi^\pi(w)$ completely describes the
action of $G_2(K)$ on
$C_c^\infty(K^\times, \mathbb{\overline{Q}}_l)$, where
$$ w =
\begin{pmatrix}
	0 & 1\\
	-1 & 0
\end{pmatrix}.$$
Here, we follow the exposition in \cite[Section
37.3]{MR2234120}. Let $\chi$ be a smooth character of
$K^\times$ and let $k$ be an integer. Define a function 
$\xi\{\chi, k\}$ in $C_c^\infty(K^\times,
\mathbb{\overline{Q}}_l)$ by setting $\xi\{\chi, k\}(x) =
\chi(x)$, for $\nu_K(x)=k$ and zero otherwise.  Recall that 
$\nu_K$ is a discrete valuation on $K^\times$.
Then we have :
\begin{equation}
\mathbb{K}_\psi^\pi(w)\xi\{\chi,k\} =
\epsilon(\chi^{-1}\pi,\psi)
\xi\{\chi^{-1}\varpi_\pi, -n(\chi^{-1}\pi,\psi)
- k\},
\end{equation}
where $\varpi_\pi$ is the central character of
$\pi$. Here $\epsilon(\pi,
\psi)$ is the Godement--Jacquet local
$\epsilon$-factor associated with a cuspidal representation
$\pi$ and some additive character $\psi$ of $F$.
	
\section{Review of Local Constants and Weil-Deligne
representations}\label{Deligne_rep_section}
\subsection{}\label{17}
Keeping the notation as in Section \ref{x}, we briefly discuss
about the Weil group and its Weil-Deligne representations. For
a reference, see \cite[Chapter 7]{MR2234120} and \cite[Chapter
4]{MR0349635}.
	
We choose a separable algebraic closure $\overline{K}$ of
$K$. Let $\Omega_K$ be the absolute Galois group
${\rm Gal}(\overline{K}/K)$ and Let
$\mathcal{I}_K$ be the inertia subgroup of
$\Omega_K$. Let $\mathcal{W}_K$ denote the Weil group of
$K$. Fix a geometric Frobenius element ${\rm
Frob}$ in $\mathcal{W}_K$. Then we have 
$$ \mathcal{W}_K =
\mathcal{I}_K \rtimes \rm Frob^{\mathbb{Z}}. $$
	
There is a natural Krull topology on the absolute Galois group
$\Omega_K$ and the inertia group
$\mathcal{I}_K$, as a subgroup of
$\Omega_K$, is equipped with the subspace topology. Let the
fundamental system of neighbourhoods of the Weil group
$\mathcal{W}_K$ be such that each neighbourhood of the
identity
$\mathcal{W}_K$ contains an open subgroup of
$\mathcal{I}_K$. Then under this topology, the Weil group
$\mathcal{W}_K$ becomes a locally compact and totally
disconnected group. If $K_1/K$ is a finite extension with
$K_1\subseteq
\overline{K}$, then the Weil group
$\mathcal{W}_{K_1}$ is considered as a subgroup of
$\mathcal{W}_K$.
	
An $R$-representation $\rho$ of $\mathcal{W}_K$ is called
unramified if $\rho$ is trivial on $I_K$. Let $\nu$ be the
unramified character of $\mathcal{W}_K$ which satisfies
$\nu(\rm Frob)$ = $q_K^{-1}$. We now define semisimple
Weil-Deligne representations of $\mathcal{W}_K$.
	
\subsection{Semisimple Weil-Deligne representation}\label{18}
A Weil-Deligne representation of $\mathcal{W}_K$ is a pair
$(\rho, U)$, where $\rho$ is a finite dimensional $R$-representation
of $\mathcal{W}_K$ and $U$ is a nilpotent endomorphism of the vector
space underlying $\rho$ and intertwining the actions of $\nu\rho$ and
$\rho$.  A Weil-Deligne representation $(\rho, U)$ of $\mathcal{W}_K$
is called semisimple if $\rho$ is semisimple as a representation of
$\mathcal{W}_K$. Note that any semisimple representation $\rho$ of
$\mathcal{W}_K$ is considered as a semisimple Weil-Deligne
representation of the form $(\rho, 0)$. For two Weil-Deligne
representations $(\rho,U)$ and $(\rho', U')$ of $\mathcal{W}_K$, let
\begin{center}
${\rm Hom}_D((\rho,U),(\rho',U'))=\{f\in{\rm
Hom}_{\mathcal{W}_K}(\rho,\rho'):f\circ U= U'\circ f\}$,
\end{center}
We say that $(\rho, U)$ and $(\rho',U')$ are isomorphic if there
exists a map $f\in{\rm Hom}_D((\rho, U)$, $(\rho', U'))$ such that $f$
is bijective. Let $\mathcal{G}^n_{ss}(K)$ be the set of all
$n$-dimensional semisimple Weil-Deligne representations of the Weil
group $\mathcal{W}_K$.
	
\subsection{Local Constants of Weil-Deligne representation}
Keep the notations as in sections (\ref{17}) and (\ref{18}). In this subsection,
we consider the local constants for $l$-adic Weil-Deligne 
representations of $\mathcal{W}_K$.
\subsubsection{L-factors.}
Let $(\rho, U)$ be an $l$-adic semisimple Weil-Deligne representation
of $\mathcal{W}_K$. Then the $L$-factor corresponding to $(\rho, U)$
is the following rational function in $X$:
\begin{center}
$L(X,(\rho, U))$ = ${\rm det}((\id-X\rho({\rm
Frob}))|_{\ker(U)^{\mathcal{I}_K}})^{-1}$.
\end{center}
	
\subsubsection{Local $\epsilon$-factors and $\gamma$-factors}
Let $\psi_K: K \rightarrow \overline{\mathbb{Q}}_l^\times$ be a
non-trivial additive character and choose a self-dual additive Haar
measure on $K$ with respect to $\psi_K$. Let $\rho$ be an $l$-adic
representation of $\mathcal{W}_K$. The epsilon factor
$\epsilon(X,\rho,\psi_K)$ of $\rho$, relative to $\psi_K$ is defined
in \cite{MR0349635}. Let $K'/K$ be a finite extension inside
$\overline{K}$. Let $\psi_{K'}$ denotes the character of $K'$, where
$\psi_{K'} = \psi_K \circ {\rm Tr}_{K'/K}$.  Then the epsilon factor
satisfies the following properties :
\begin{enumerate}
\item If $\rho_1$ and $\rho_2$ are two $l$-adic
representations of $\mathcal{W}_{K}$, then
$$ \epsilon(X,\rho_1 \oplus \rho_2,\psi_K) = \epsilon(X,\rho_1,\psi_K)
\epsilon(X,\rho_2,\psi_K). $$
\item $\rho$ is an $l$-adic representation of $\mathcal{W}_{K'}$,
then
\begin{equation}\label{60}
  \frac{\epsilon\big(X, {\rm ind}_{\mathcal{W}_{K'}}^{\mathcal{W}_{K}}
    (\rho),
    \psi_K\big)}{\epsilon(X, \rho, \psi_{K'})}= 
  \left\{\frac{\epsilon\big(X, {\rm ind}_{\mathcal{W}_{K'}}^{\mathcal{W}_{K}}(1_{K'}),
      \psi_K\big)}{\epsilon(X, 1_{K'},\psi_{K'})}\right\}^{{\rm dim}(\rho)},
\end{equation}
where $1_{K'}$ denotes the trivial character of
$\mathcal{W}_{K}$.
\item If $\rho$ is an $l$-adic representation of $\mathcal{W}_{K}$,
then
\begin{equation}\label{61}
\epsilon\big(X, \rho,\psi_K\big) \epsilon\big(q_K^{-1}X^{-1}, 
\rho^\vee,\psi_K\big) = \rm det\big(\rho(-1)\big),
\end{equation}
where $\rho^\vee$ denotes the dual of the representation $\rho$.
\item For an $l$-adic representation $\rho$ of
$\mathcal{W}_K$, there exists an integer
$n(\rho,\psi_K)$ for which
$$ \epsilon(X,\rho,\psi_K) =
(q_K^{\frac{1}{2}}X)^{n(\rho,\psi_K)}\epsilon(\rho,\psi_K). $$
\end{enumerate}
Now for an $l$-adic semisimple Weil-Deligne representation
$(\rho, U)$, the $\epsilon$-factor is defined as
$$ \epsilon\big(X, (\rho, U), \psi_K\big) = \epsilon(X, \rho,
\psi_K)\frac{L(q_K^{-1}X^{-1},\rho^\vee)}{L(X,\rho)}
\frac{L(X,(\rho,U))}{L(q_K^{-1}X^{-1},(\rho, U)^\vee)}, $$ where
$(\rho, U)^\vee = (\rho^\vee, -U^\vee)$. Set
$$ \gamma(X,(\rho, U),\psi_K) =
\epsilon(X,(\rho,U),\psi_K)\frac{L(X,(\rho,U))}
{L(q_K^{-1}X^{-1},(\rho,U)^\vee)}. $$ The element
$\gamma(X, (\rho, U), \psi_K)$ is called the $\gamma$-factor of the
Weil-Deligne representation $(\rho, U)$.
	
Now we state a result \cite[Proposition 5.11]{MR4311563} which
concerns the fact that the $\gamma$-factors are compatible with
reduction modulo $l$. For $P \in \mathbb{\overline{Z}}_l[X]$, we
denote by $r_l(P) \in \mathbb{\overline{F}}_l[X]$ the polynomial
obtained by reduction mod-$l$ to the coefficients of $P$. For
$Q \in \mathbb{\overline{Z}}_l[X]$, such that $r_l(Q)\not= 0$, we set
$r_l(P/Q) = r_l(P)/r_l(Q)$.
\begin{proposition}\label{62}
  Let $\rho$ be an integral $l$-adic semisimple representation of
  $\mathcal{W}_K$. Then
$$ r_l\big(\gamma(X,\rho,\psi_K)\big) =
\gamma\big(X, r_l(\rho),\overline{\psi}_K\big), $$
where $\overline{\psi}_K$ is the reduction mod-$l$ of
$\psi_K$.
\end{proposition}
We end this subsection with a lemma which will be needed later in the
proof of Theorem (\ref{intro_n_thm}).
\begin{lemma}\label{i}
Let $E/F$ be a cyclic Galois extension of prime degree $l$ and
assume $l \not= 2$.  Let $\rho$ be an $l$-adic
representation of $\mathcal{W}_E$ of even dimension. Then
\begin{center}
$\epsilon(X, \rho, \psi_E) = \epsilon\big(X, {\rm
ind}_{\mathcal{W}_E}^{\mathcal{W}_F}(\rho),\psi_F\big)$.
\end{center}
\end{lemma}
\begin{proof}
Let
$\mathcal{C}_{E/F}(\psi_F) = \dfrac{\epsilon\big(X, {\rm
ind}_{\mathcal{W}_E}^{\mathcal{W}_F}(1_E),
\psi_F\big)}{\epsilon(X, 1_E, \psi_E)}$, where $1_E$ denotes the
trivial character of $\mathcal{W}_E$. Then
$\mathcal{C}_{E/F}(\psi_F)$ is independent of X (see \cite[Corollary
30.4, Chapter 7]{MR2234120}). Using the equality (\ref{60}), we get
\begin{center}
$\dfrac{\epsilon\big(X, {\rm
ind}_{\mathcal{W}_E}^{\mathcal{W}_F}(\rho),
\psi_F\big)}{\epsilon(X, \rho, \psi_E)}
= (\mathcal{C}_{E/F}\big(\psi_F)\big)$$^{\rm dim\rho}$.
\end{center}
In view of the functional equation (\ref{61}), we have
\begin{center}
$\mathcal{C}_{E/F}(\psi_F)^2 = \xi_{E/F}(-1)$,
\end{center}
where $\xi_{E/F} = {\rm det}\big({\rm
ind}_{\mathcal{W}_E}^{\mathcal{W}_F}(1_E)\big)$, a quadratic
character of $\mathcal{W}_F$. Note that
$\xi_{E/F}$ is a character of the quotient group
$\mathcal{W}_F/\mathcal{W}_E$--which is a cyclic group of order
$l$, and this implies that $\xi_{E/F}^l = 1$. Since
$l$ is odd, we get that $\xi_{E/F} =
1_F$, the trivial character of $\mathcal{W}_F$. Hence the lemma.
\end{proof} 
\subsection{Local constants of \texorpdfstring{$p$}{}-adic
representations}\label{int_1}
Following the notations as in Section (\ref{rr}), we now define
the $L$-factors and $\gamma$-factors for irreducible
$R$-representations of $G_n(K)$. For details, see \cite{MR3595906}. 
Let $\pi$ be an $R$-representation of Whittaker type of
$G_n(K)$ and let $\pi'$ be an
$R$-representation of Whittaker type of $G_{n-1}(K)$. Let $W \in
\mathbb{W}(\pi,\psi_K)$ and $W' \in
\mathbb{W}(\pi',\psi_K^{-1})$. The function $W
\begin{pmatrix}
	g & 0\\
	0 & 1
\end{pmatrix}
W'(g)$ is compactly supported on $Y_K^r$ \cite[Proposition 3.3]{MR3595906}.
Then the following integral
$$ 
c^K_r(W,W') = \int_{Y_K^r} W
\begin{pmatrix}
	g & 0\\
	0 & 1
\end{pmatrix}
W'(g) \,dg, $$ 
is well defined for all $r \in \mathbb{Z}$, and vanishes 
for $r << 0$. In this paper, we deal with base change where two different
$p$-adic fields are involved. So to avoid confusion, we use the
notation $c^K_r(W,W')$ instead of the notation
$c_r(W,W')$ used in \cite[Proposition 3.3]{MR3595906} for these
integrals on $Y_K^r$. Now consider the functions
$\widetilde{W}$ and $\widetilde{W'}$, defined as
$$ \widetilde{W}(x) = W\big(w_n(x^t)^{-1}\big) $$ 
and 
$$ \widetilde{W'}(g) = W'\big(w_{n-1}(g^t)^{-1}\big), $$
for all $x \in G_n(K)$, $g \in
G_{n-1}(K)$. Then making a change of variables, we have the following relation:
\begin{equation}\label{int_2}
c^K_r(\widetilde{W},\widetilde{W'}) =
c^K_{-r}\big(\pi(w_n)W,\pi'(w_{n-1})W'\big).
\end{equation}
Let $I(X,W,W')$ be the following power series:
\begin{equation}\label{PS}
I(X,W,W') = \sum_{r \in \mathbb{Z}}c^K_k(W,W') 
q_K^{r/2} X^r \in R((X)) . 
\end{equation}
Note that $I(X,W,W')$ is a rational function in $X$ (see
\cite[Theorem 3.5]{MR3595906}).
	
\subsubsection{$L$ -factors}
Let $\pi$ and $\pi'$ be two $R$-representations of Whittaker type of
$G_n(K)$ and $G_{n-1}(K)$ respectively. Then the $R$-submodule spanned
by $I(X,W,W')$ as $W$ varies in $\mathbb{W}(\pi,\psi_K)$ and $W'$
varies in $\mathbb{W}(\pi',\psi_K^{-1})$, is a fractional ideal of
$R[X,X^{-1}]$ and it has a unique generator which is an Euler factor
denoted by $L(X,\pi,\pi')$. The generator $L(X,\pi,\pi')$ called the
$L$-factor associated to $\pi$, $\pi'$ and $\psi_K$.
\begin{remark}
If $\pi$ and $\pi'$ are $l$-adic representations of Whittaker type
of $G_n(K)$ and $G_{n-1}(K)$ respectively, then
$1/L(X,\pi,\pi') \in \mathbb{\overline{Z}}_l[X]$.
\end{remark}
We conclude this section with a theorem \cite[Theorem 4.3,]{MR3595906}
which describes $L$-factors of cuspidal representations.
\begin{theorem}\label{k}
\rm Let $\pi_1$ and $\pi_2$ be two cuspidal
$R$-representations of $G_n(K)$ and
$G_m(K)$ respectively. Then
$L(X,\pi_1,\pi_2)$ is equal to 1, except in the following case :
$\pi_1$ is banal in the sense of \cite{MR3178433} and $\pi_2 \simeq
\chi \pi_1^\vee$ for some unramified character $\chi$ of $K^\times$.
\end{theorem}
In the proof of Theorem(\ref{intro_n_thm}), we only consider the case
when $m=n-1$, and by the above theorem the
$L$-factor
$L(X,\pi_1,\pi_2)$ associated with the cuspidal
$R$-representations $\pi_1$ and $\pi_2$ is equal to 1.
	
\subsubsection{Functional Equations and Local
\texorpdfstring{$\gamma$}{}-factors}\label{w}
Let $\pi$ and $\pi'$ be two $R$-representations of
Whittaker type of $G_n(K)$ and $G_{n-1}(K)$
respectively. Then there is an invertible element
$\epsilon(X,\pi,\pi',\psi_K)$ in $R[X,X^{-1}]$ such that
for all $W\in\mathbb{W}(\pi,\psi_K)$,
$W' \in \mathbb{W}(\pi',\psi_K^{-1})$, we have the
following functional equation :
$$ \dfrac{I(q_K^{-1}X^{-1},\widetilde{W},
\widetilde{W'})}{L(q_K^{-1}X^{-1}, \widetilde{\pi},
\widetilde{\pi'})}=\omega_{\pi'}(-1)^{n-2} \epsilon(X,
\pi,\pi',\psi_K)\dfrac{I(X,W,W')}{L(X,\pi,\pi')}, $$ 
where $\widetilde{W}$ is defined as in (\ref{int_1}) and
$\omega_{\pi'}$ denotes the central character of the representation
$\pi'$. We call $\epsilon(X,\pi,\pi',\psi_K)$ the local
$\epsilon$-factor associated to $\pi$, $\pi'$ and $\psi_K$. Moreover,
if $\pi$ and $\pi'$ be $l$-adic representations of Whittaker type of
$G_n(K)$ and $G_{n-1}(K)$ respectively, then the factor
$\epsilon(X,\pi,\pi',\psi_K)$ is of the form $cX^k$ for a unit
$c \in \mathbb{\overline{Z}}_l$. In particular, there exists an
integer $n(\pi,\pi',\psi_K)$ such that
\begin{equation}\label{degree}
\epsilon(X,\pi,\pi',\psi_K)=(q_K^{\frac{1}{2}}X)
^{n(\pi,\pi',\psi_K)}
\epsilon(\pi,\pi',\psi_K).
\end{equation} 
Now the local $\gamma$-factor associated with
$\pi$, $\pi'$ and $\psi$ is defined as:
\begin{center}
$\gamma(X,\pi,\pi',\psi_K) = \epsilon(X,\pi,\pi',\psi_K)
\dfrac{L(q_K^{-1}X^{-1}, \widetilde{\pi},
\widetilde{\pi'})}{L(X,\pi,\pi')}$.
\end{center}
	
\subsubsection{Compatibility with reduction modulo $l$}\label{prop}
Let $\tau$ and $\tau'$ be two $l$-modular representations of Whittaker
type of $G_n(K)$ and $G_{n-1}(K)$ respectively.  Let $\pi$ and $\pi'$
be the respective Whittaker lifts of $\tau$ and $\tau'$.  Then
$$ L(X,\tau,\tau') | r_l\big(L(X,\pi,\pi')\big) $$
and 
$$ 
r_l\big((\gamma(X,\pi,\pi',\psi_K)\big)
= \gamma(X,\tau,\tau',\overline{\psi}_K).
$$ 
For details, see \cite[Section 3.3]{MR3595906}. 
	
\subsubsection{Generic part of mod-\texorpdfstring{$l$}{} reduction}
Let $\pi$ be an integral $l$-adic generic representation of
$G_n(K)$. The mod-$l$-reduction of $\pi$, denoted by $r_l(\pi)$, has a
unique generic component and it is denoted by $J_l(\pi)$ (see
\cite[Section 1.8.4]{MR1821157}). Let $\sigma$ be an $l$-adic generic
representation of $G_{n-1}(K)$. Now, the functional equation for the
pair $(J_l(\pi), J_l(\sigma))$ gives
$$
I(q_K^{-1}X^{-1},\widetilde{W},\widetilde{W'})=
\varpi_{\sigma}(-1)^{n-2}\gamma(X,J_l(\pi),J_l(\sigma),
\overline{\psi}_K)I(X,W,W'),
$$
for all $W\in\mathbb{W}(J_l(\pi),\overline{\psi}_K)$ and
$W'\in\mathbb{W} (J_l(\sigma),\overline{\psi}_K^{-1})$. Let us
consider the following commutative diagram:
$$
\xymatrix{
\mathbb{W}^0(\pi,\psi_K) \ar[dd]_{{\rm Res}_{P_n(K)}}
\ar[rr]^{\Lambda_\pi}  &&  {\rm Ind}_{N_n(K)}^{G_n(K)}
\overline{\Theta}_K \ar[dd]^{{\rm Res}_{P_n(K)}} \\\\
\mathbb{K}^0(\pi,\psi_K) \ar[rr]_{\lambda_\pi} && {\rm
Ind}_{N_n(K)}^{P_n(K)}\overline{\Theta}_K }
$$
Note that the restriction to $P_n(K)$ map on $W^0(\pi, \psi_K)$ is an
isomorphism.  Here $\Lambda_\pi$ and $\lambda_\pi$ are the pointwise
mod-$l$ reduction maps. Since $\mathcal{K}^0(\psi_K)$ is contained in
$\mathbb{K}^0(\pi,\psi_K)$ and $\lambda_\pi$ maps
$\mathcal{K}^0(\psi_K)$ onto $\mathcal{K}(\overline{\psi}_K)$, the
$P_n(K)$-equivariant map $\lambda_\pi$ is non-zero. It then follows
from the commutativity of the above diagram that $\Lambda_\pi$ is
non-zero. Since $J_l(\pi)$ is the unique generic subquotient of
$r_l(\pi)$, the image of $\Lambda_\pi$ contains the Whittaker space
$\mathbb{W}(J_l(\pi),\overline{\psi}_K)$. Similarly, the image of
$\Lambda_\sigma$ contains $\mathbb{W}(J(\sigma),\overline{\psi}_K)$.
Let $U$ (resp. $U'$) be an element of $\mathbb{W}^0(\pi,\psi_K)$
(resp. $\mathbb{W}^0(\sigma,\psi_K)$) such that $\Lambda_\pi(U)=W$
(resp.  $\Lambda_\tau(U')=W'$). From the functional equation for the
pair $(\pi,\sigma)$, we get the following relation
$$
I(q_K^{-1}X^{-1},\widetilde{U},\widetilde{U'}) =
\varpi_\sigma(-1)^{n-2}\gamma(X,\pi,\sigma,\psi_K)I(X,U,U').
$$
After reducing the above equality modulo-$l$, we have
$$
I(q_K^{-1}X^{-1},\widetilde{W},\widetilde{W'}) = \varpi_\sigma(-1)^{n-2}
r_l(\gamma(X,\pi,\sigma,\psi_K)) I(X,W,W'),
$$
Thus, we get that
\begin{equation}\label{gamma_factor_generic_part}
  r_l(\gamma(X, \pi, \sigma, \psi_K)) =
  \gamma(X, J_l(\pi), J_l(\sigma),\overline{\psi}_K).
\end{equation}
	
\section{Local Langlands Correspondence}\label{LLC_section}
\subsection{The \texorpdfstring{$l$}{}-adic local Langlands
correspondence}\label{LLC}
In this subsection, we recall the $l$-adic local Langlands
correspondence. Keep the notation as in section (\ref{x}). Let
$\psi_K$ be a non-trivial additive character of $K$. Recall that
local Langlands correspondence over $\overline{\mathbb{Q}}_l$ is
the bijection
\begin{center}
$\Pi_K:{\rm Irr}\big(GL_n(K),\mathbb{\overline{Q}}_l\big)
\longrightarrow \mathcal{G}_{ss}^n(K)$
\end{center}
such that
$$ \gamma(X,\sigma\times\sigma',\psi_K) =
\gamma(X,\Pi_K(\sigma)\otimes\Pi_K(\sigma'),\psi_K) $$
and
$$ L(X,\sigma\times\sigma') = L(X,\Pi_K(\sigma)\otimes\Pi_K(\sigma')),
$$
for all $\sigma\in{\rm Irr}(G_n(K),\mathbb{\overline{Q}}_l)$,
$\sigma'\in{\rm Irr}(G_m(K),\mathbb{\overline{Q}}_l)$. Moreover, the
set of all cuspidal $l$-adic representations of $GL_n(K)$ is mapped
onto the set $n$-dimensional irreducible $l$-adic representations of
the Weyl group $\mathcal{W}_K$ via the bijection $\Pi_K$ (see
\cite{harris_taylor}, \cite{henniart_une_preuve} or
\cite{scholze_llc}). Note that the classical local Langlands
correspondence is a bijection between
${\rm Irr}({\rm GL}_n(K), \mathbb{C})$ and the isomorphism classes of
$n$-dimensional, complex semisimple Weil--Deligne representations. To
get a correspondence over $\overline{\mathbb{Q}}_l$, one twists the
original correspondence by the character $\nu^{(1-n)/2}$. For details
see \cite[Conjecture 4.4, Section 4.2]{clozel_motifs}, \cite[Section
7]{henniart_bordeaux} and for $n=2$ see \cite[Theorem
35.1]{MR2234120}.
	
\subsection{Local base Change for the extension
\texorpdfstring{$E/F$}{}}\label{base_change}
Now we recall local base change for a cyclic extension of a $p$-adic
field. The base change operation on irreducible smooth representations
of ${\rm GL}_n(F)$ over complex vector spaces is characterised by
certain character identities (see \cite[Chapter 3]{MR1007299}). Let us
recall the relation between $l$-adic local Langlands correspondence
and local base change for ${\rm GL}_n$. Let $\pi_F$ be an $l$-adic
irreducible smooth representation of ${\rm GL}_n(F)$. Let
$(\rho_F, U)$ be a  semisimple Weil-Deligne representation such that
$\Pi_F(\pi_F) = \rho_F$, where $\Pi_F$ is the $l$-adic local Langlands
correspondence as described in the previous section. Let $\pi_E$ be
the $l$-adic irreducible representation of $GL_n(E)$ such that
\begin{center}
${\rm res}_{\mathcal{W}_E}\big(\Pi_F(\pi_F)\big) \simeq
\Pi_E(\pi_E)$.
\end{center}
The representation $\pi_E$ is the base change of $\pi_F$. Note that in
this case $\pi_E\simeq\pi_E^\gamma$, for all $\gamma \in\Gamma$.
\subsubsection{Base change for $L(\Delta)$}\label{bc_gen}
Let $k$ be a positive integer. Let
$\Delta =\{\tau_F,\tau_F\nu_F,...,\tau_F\nu_F^{k-1}\}$ be a segment,
where $\tau_F$ is an $l$-adic cuspidal representation of
$G_m(F)$. Consider the generic representation $\mathcal{L}(\Delta)$ of
$G_{km}(F)$. Then
$$ \Pi_F\big(\mathcal{L}(\Delta)\big)=\Pi_F(\tau_F)
\otimes {\rm Sp}_F(k), $$
where ${\rm Sp}_F(k)$ is the semisimple Weil-Deligne representation of
$\mathcal{W}_F$, defined as in \cite[Section 31, Example
31.1]{MR2234120}. If $l$ does not divide $m$, then there exists a
cuspidal representation $\tau_E$ of $G_m(E)$ such that $\tau_E$ is a
base change of $\tau_F$ that is,
$$ {\rm res}_{\mathcal{W}_E}\big(\Pi_F(\tau_F)\big) = \Pi_E(\tau_E). $$
Then we have
$$ {\rm res}_{\mathcal{W}_E}
\big(\Pi_F(\mathcal{L}(\Delta))\big)= \Pi_E(\tau_E) \otimes {\rm
Sp}_E(k)=\Pi_E\big(\mathcal{L}(D)\big), 
$$ 
where $D$ is the segment
$\{\tau_E,\tau_E\nu_E,...,\tau_E\nu_E^{k-1}\}$. Hence, it follows that
the generic representation $\mathcal{L}(D)$ of $G_{km}(E)$ is a base
change of $\mathcal{L}(\Delta)$. Next, we prove a lemma
about base change lifting and integrality.

\begin{lemma}\label{bc_integral}
  Let $\pi_F$ be an irreducible $l$-adic representation of $G_n(F)$,
  and let $\pi_E$ be the base change lifting of $\pi_F$ to
  $G_n(E)$. Then $\pi_F$ is integral if and only if the base change
  lifting $\pi_E$ is integral.
\end{lemma}
\begin{proof}
  Let $\pi$ be an irreducible $l$-adic smooth representation of
  $G_n(K)$. Let ${\rm scs}(\pi)$ be the supercuspidal support of $\pi$
  (see \cite[III.3.]{Vigneras_Induced} for the definition). The
  representation $\pi$ is integral if and only if ${\rm scs}(\pi)$ is
  integral (see \cite[Section 1.4]{MR1821157} and for general
  reductive groups see \cite[Corollary 1.6]{dat2024finiteness} for a
  reference).  The representation ${\rm scs}(\pi)$ is integral if and
  only if the central character of ${\rm scs}(\pi)$ is integral.
  
  Assume that $\pi_F$ is integral. Let $(\rho_F,U)$ be the $l$-adic,
  semisimple Weil-Deligne representation of $\mathcal{W}_F$ associated
  with $\pi_F$ under the local Langlands correspondence (LLC)
  $\Pi_F$. Under LLC, the representation ${\rm scs}(\pi_F)$
  corresponds to the $\mathcal{W}_F$-representation $\rho_F$.  Since
  the determinant character of each irreducible component of $\rho_F$
  is integral, we get that $\rho_F$ is integral. This implies that the
  restriction ${\rm res}_{\mathcal{W}_E}(\rho_F)$ is integral. Under
  LLC, the supercuspidal support of $\pi_E$ corresponds to the
  restriction ${\rm res}_{\mathcal{W}_E}(\rho_F)$.  Thus, the
  supercuspidal support of $\pi_E$ is integral and we get that $\pi_E$
  is integral.

  Conversely assume that $\pi_E$ is integral. Let
  $(\rho_E,U_E)$ be the $l$-adic, semisimple Weil-Deligne
  representation of $\mathcal{W}_E$ associated with $\pi_E$ under
  LLC. Since the supercuspidal support of $\pi_E$ is integral, the
  $\mathcal{W}_E$-representation $\rho_E$ (which is
  ${\rm res}_{\mathcal{W}_E}(\rho_F)$) is integral. Let $\mathcal{L}_E$
  be a $\mathcal{W}_E$-stable lattice in $\rho_E$. Then,
$$ \sum_{x\in \mathcal{W}_F/\mathcal{W}_E} \rho_F(x)\mathcal{L}_E. $$ 
is a $\mathcal{W}_F$-stable lattice in $\rho_F$. Thus, the
representation $\rho_F$ is also integral, which implies that the
supercuspidal support ${\rm scs}(\pi_F)$ is integral. Hence, $\pi_F$
is integral.
\end{proof}
	
\section{Tate Cohomology}\label{Tate_cohomology}
In this section, we recall Tate cohomology and some useful results on
$\Gamma$-equivariant $l$-sheaves of $\Lambda$-modules on an $l$-space $X$
equipped with an action of $\Gamma$. We refer to \cite[Section
3]{MR3432583} for details.
\subsection{}
Fix a generator $\gamma$ of $\Gamma$. Let $M$ be a
$\Lambda[\Gamma]$-module, and let $T_\gamma$ be the
automorphism of $M$ defined by
\begin{center}
$T_\gamma(m)=\gamma.m$, for $\gamma\in\Gamma, m\in M$.
\end{center}
Let
$N_\gamma={\rm id}+T_\gamma+T_{\gamma^2}+....+T_{\gamma^{l-1}}$
be the norm operator. The Tate cohomology groups $\widehat{H}^0(M)$
and $\widehat{H}^1(M)$ are defined as :
\begin{center}
$\widehat{H}^0(M) = \dfrac{{\rm ker}({\rm id}-T_\gamma)}{{\rm
Im}(N_\gamma)}$, $\widehat{H}^1(M) = \dfrac{{\rm
ker}(N_\gamma)}{{\rm Im}({\rm id}-T_\gamma)}$.
\end{center}
\subsection{Tate Cohomology of sheaves on
\texorpdfstring{$l$}{}-spaces}
Let $X$ be an $l$-space equipped with an action of a finite group
$\langle\gamma\rangle$ of order $l$. Let $\mathcal{F}$
be an $l$-sheaf of $\Lambda$ modules 
on $X$. Write $\Gamma_c(X;\mathcal{F})$ for the space of compactly
supported sections of $\mathcal{F}$. In particular, if
$\mathcal{F}$ is the constant sheaf with stalk
$\Lambda$,
then $\Gamma_c(X;\mathcal{F}) =
C_c^\infty(X, \Lambda)$. The assignment
$\mathcal{F} \mapsto \Gamma_c(X;\mathcal{F})$ is a
covariant exact functor. If $\mathcal{F}$ is
$\gamma$-equivariant, then $\gamma$ can be regarded as a
map of sheaves
$\mathcal{F}|_{X^\gamma}\rightarrow
\mathcal{F}|_{X^\gamma}$ and the Tate cohomology is
defined as 
\begin{center}
$\widehat{H}^0(\mathcal{F}|_{X^\gamma}):=
{\rm ker}(1-\gamma)/{\rm Im}(N)$,\\
$\widehat{H}^1(\mathcal{F}|_{X^\gamma}):=
{\rm ker}(N)/{\rm Im}(1-\gamma)$.
\end{center}
A compactly supported section of
$\mathcal{F}$ can be restricted to a compactly supported section of
$\mathcal{F}|_{X^\gamma}$. The following result is often useful in
calculating Tate cohomology groups.
\begin{proposition}[Treumann-Venkatesh, 
\cite{MR3432583}]\label{tr}
The restriction map induces an isomorphism of the following spaces:
\begin{center}
$\widehat{H}^i(\Gamma_c(X;\mathcal{F}))\longrightarrow
\Gamma_c(X^\gamma;\widehat{H}^i(\mathcal{F}))$
for $i=0,1$.
\end{center}
\end{proposition}
\subsubsection{}\label{section_kirrilov_rep}
The above proposition is very useful to compute the Tate cohomology
groups 
of compactly induced representations. For instance, the following argument is 
used at many places in the paper. Let $E/F$ be a 
Galois extension of degree $l$ with Galois group $\Gamma$.
Let $\psi_F:F\rightarrow 
\overline{\mathbb{Z}_l}^\times $ be an additive character
and let $\psi_E$ be the character 
$\psi_F\circ {\rm Tr}_{E/F}$,
where ${\rm Tr}_{E/F}$ is the trace function. Let $\Theta_E$
and $\Theta_F$ be the non-degenerate characters of $N_n(E)$ 
and $N_n(F)$ associated with $\psi_F$ and $\psi_E$ 
respectively (see Subsection \ref{rr}). We will use the notations 
$\overline{\psi}_E$, 
$\overline{\psi}_F$, $\overline{\Theta}_E$, $\overline{\Theta}_F$
for the respective mod-$l$ reductions. 
Recall the notation $\mathcal{K}(\psi_K)$ for  the
compact induction $\ind_{N_n(K)}^{P_n(K)}\psi_K$. Note that 
the Galois group $\Gamma$ acts on the representation 
$\mathcal{K}(\overline{\psi}_E)$, by setting 
$$(\gamma f)(x)=f(\gamma^{-1}(x)),\ \gamma\in \Gamma,\  x\in P_n(E),\
f\in \mathcal{K}(\overline{\psi}_E).$$
The restriction to $P_n(F)$ map:
$$\mathcal{K}(\overline{\psi}_E)\rightarrow 
\mathcal{K}(\overline{\psi}_F^l),\ f\mapsto {\rm res}_{P_n(F)} f$$
factorizes through 
\begin{equation}\label{kirillov_mod_tate_commutes}
    \widehat{H}^0(\mathcal{K}(\overline{\psi}_E))
    \rightarrow \mathcal{K}(\overline{\psi}_F^l).
\end{equation}
We set $Y_K=P_n(K)/N_n(K)$. Note that the pointed set
$H^1(\Gamma, N_n(E))$ is trivial and hence $Y_E^\Gamma=Y_F$.  Applying
Proposition \ref{tr} for the case where $X=Y_E$ and $\mathcal{F}$
equals the sheaf associated with the induced representation
$\ind_{N_n(E)}^{P_n(E)}\overline{\Theta}_E$, we get that the map
\eqref{kirillov_mod_tate_commutes} is an isomorphism.  Proposition
\ref{tr} also shows that $\widehat{H}^1(\mathcal{K}^0(\psi_E))$ is
trivial. Here, $\mathcal{K}^0(\psi_E)$ is the space of
$\overline{\mathbb{Z}}_l$-valued functions in the $l$-adic
representation $\ind_{N_n(E)}^{P_n(E)}\psi_E$.

\subsection{Comparison of integrals of smooth functions}\label{ur}
The group $\Gamma =\langle\gamma \rangle$ acts on the space
$X_E=G_{n-1}(E)/N_{n-1}(E)$ and hence its action on the space
$C_c^\infty(X_E, \mathbb{\overline{F}}_l)$ is given by the following equality: 
\begin{center}
$(\gamma.\phi)(x):=\phi(\gamma^{-1}x)$, for all $x\in X_E$, and 
$\phi\in C_c^\infty(X_E,\mathbb{\overline{F}}_l)$.
\end{center}
Let $C_c^\infty(X_E, \overline{\mathbb{F}}_l)^\Gamma$ be the space of
all $\Gamma$-invariant functions in
$C_c^\infty(X_E, \overline{\mathbb{F}}_l)$.  We end this section with
a proposition comparing the integrals on the spaces $X_E$ and $X_F$.
\begin{proposition}\label{1}
Let $d\mu_E$ and $d\mu_F$ be Haar measures on $X_E$ and $X_F$
respectively. Then, there exists a non-zero scalar
$c\in \overline{\mathbb{F}}_l$ such that
$$ \int_{X_E}\phi \,d\mu_E = c \int_{X_F}\phi \,d\mu_F, $$
for all
$\phi\in C_c^\infty(X_E,\mathbb{\overline{F}}_l)^\Gamma$.
\end{proposition}
\begin{proof}
Since $N_{n-1}(E)$ is stable under the action of $\Gamma$ on
$G_{n-1}(E)$, we have the following long exact sequence of
non-abelian cohomology \cite[Chapter VII, Appendix]{MR1324577}:
$$ 
0 \longrightarrow N_{n-1}(E)^\Gamma \longrightarrow
G_{n-1}(E)^\Gamma \longrightarrow X_E^\Gamma \longrightarrow
H^1\big(\Gamma; N_{n-1}(E)\big) \longrightarrow H^1\big(\Gamma;
G_{n-1}(E)\big). 
$$ 
Since $H^1\big(\Gamma; N_{n-1}(E)\big) = 0$, we
get from the above exact sequence that
$$ X_E^\Gamma \simeq X_F. $$
Since $X_F$ is closed in $X_E$, we have the following exact sequence
of $\Gamma$-modules :
\begin{equation}\label{s}
0 \longrightarrow C_c^\infty\big(X_E \setminus X_F,
\mathbb{\overline{F}}_l\big)\longrightarrow C_c^\infty(X_E,
\mathbb{\overline{F}}_l)\longrightarrow
C_c^\infty(X_F,\mathbb{\overline{F}}_l)\longrightarrow 0.
\end{equation}
Now, the action of $\Gamma$ on $X_E \setminus X_F$ is
free, and 
it follows from Proposition \ref{tr} that
\begin{equation}\label{r}
H^1\big(\Gamma,
C_c^\infty(X_E \setminus X_F,\mathbb{\overline{F}}_l)\big) = 0.
\end{equation}
Using (\ref{s}) and (\ref{r}), we get the following exact sequence :
\begin{center}
$0 \longrightarrow C_c^\infty\big(X_E \setminus X_F,
\mathbb{\overline{F}}_l\big)^\Gamma \longrightarrow
C_c^\infty(X_E, \mathbb{\overline{F}}_l)^\Gamma
\longrightarrow
C_c^\infty(X_F,\mathbb{\overline{F}}_l)\longrightarrow 0$.
\end{center}
Again the free action of $\Gamma$ on $X_E\setminus X_F$ gives
a fundamental domain $U$ such that
$X_E\setminus X_F=\bigsqcup_{i=0}^{l-1}\gamma^iU$, and we
have
$$ \int_{X_E\setminus X_F}\phi \,d\mu_E = 
l\sum_{i=0}^{l-1}\int_U\phi \, d\mu_E = 0,
$$ 
for all
$\phi \in C_c^\infty\big(X_E \setminus X_F,
\mathbb{\overline{F}}_l\big)^\Gamma$.  Therefore the linear
functional $d\mu_E$ induces a $G_{n-1}(F)$-invariant linear
functional on $C_c^\infty(X_F, \mathbb{\overline{F}}_l)$ and
we have
$$ \int_{X_E}\phi \,d\mu_E = c \int_{X_F}\phi \,d\mu_F, $$
for some scalar $c$. Now we will show that $c \not= 0$. By
\cite[Chapter 1, Section 2.8]{MR1395151}, we have a surjective map
$\Psi : C_c^\infty(G_n(E), \overline{\mathbb{F}}_l)
\longrightarrow C_c^\infty(X_E, \overline{\mathbb{F}}_l)$,
defined by
$$ \Psi(f)(g) := \int_{N_n(E)}f(ng) \,dn, $$
for all $f \in C_c^\infty(G_n(E), \overline{\mathbb{F}}_l)$,
where $dn$ is a Haar measure on $N_n(E)$. Then there exists a
$\Gamma$-invariant compact open subgroup $I \subseteq G_n(E)$
such that $\Psi(1_I) \not= 0$, where $1_I$ denotes the
characteristic function on $I$. So the Haar measure $d\mu_E$
is non-zero on the space
$C_c^\infty(X_E, \overline{\mathbb{F}}_l)^\Gamma$ and this
implies that $c \not= 0$. Hence the proposition follows.
\end{proof}
	
\begin{remark}\label{rmk_int2}
\rm Keep the notations and hypothesis in Proposition \ref{1}.  From
now, the Haar measures $d\mu_E$ and $d\mu_F$ on $X_E$ and $X_F$
respectively, are chosen so as to make $c = 1$. Then we have
$$ \int_{X_E}\phi \,d\mu_E =  \int_{X_F}\phi \,d\mu_F. $$
Moreover, if $e$ is the ramification index of the extension $E$
over $F$, then for all $r\notin\{te : t \in \mathbb{Z}\}$, we
have
$$ \int_{(X_E^r)^\Gamma}\phi\,d\mu_F = 0 $$
and for all $r\in\{te : t\in\mathbb{Z}\}$, we have
$$ \int_{(X_E^r)^\Gamma}\phi\,d\mu_F = \int_{X_F^{\frac{r}{e}}}\phi\,d\mu_F.
$$
\end{remark}	
	
\subsection{Finiteness of Tate cohomology}
In this part, we prove some results on finiteness of the Tate
cohomology of finite length representations of ${\rm GL}_n$. First,
we introduce some notations. Let $U_n(K)$ be the subgroups of
$G_n(K)$, given by
$$ U_n(K)=\Big\{\begin{pmatrix}
	1 & C\\
	0 & 1
\end{pmatrix}: C \in K^{n-1}\Big\}$$ respectively. Note that
$U_n(K)$ is contained in the mirabolic subgroup $P_n(K)$. We use
the short notation $Z_{K,n}$ to denote the coset space 
$P_n(K)/P_{n-1}(K)U_n(K)$. Let ${\rm Rep}_R(G)$ denote the 
category of smooth
$R$-representations of a locally profinite group $G$, where $R$
denotes either $\mathcal{K}$ or $\overline{\mathbb{Q}}_l$ or
$\overline{\mathbb{F}}_l$. Then we have four fundamental
functors:
$$
\Psi^-:{\rm Rep}_R(P_n)\rightarrow {\rm Rep}_R(G_{n-1}),
\Psi^+: {\rm Rep}_R(G_{n-1})\rightarrow {\rm Rep}_R(P_n)
$$
$$
\Phi^-:{\rm Rep}_R(P_n)\rightarrow {\rm Rep}_R(P_{n-1}),
\Phi^+:{\rm Rep}_R(P_{n-1})\rightarrow {\rm Rep}_R(P_n).
$$
For the definitions of the functors $\Phi^{\pm}$ and $\Psi^{\pm}$, 
see \cite[Section 3]{MR579172} for $R=\overline{\mathbb{Q}}_l$, and
\cite[Chapter III, Section 1]{MR1395151} for
$R=\overline{\mathbb{F}}_l$.
	
\subsubsection{}
Let $\tau$ be a smooth $R$-representation of $P_n(K)$. The $m$-th
derivative of $\tau$, denoted by $\tau^{(m)}$, is defined as the
representation $\Psi^-(\Phi^-)^{m-1}(\tau)$ of $G_{n-m}(K)$. There is
a functorial filtration on $\tau$, given by
$$
0\subseteq\tau_{n}\subseteq\tau_{n-1}\subseteq\dots
\subseteq\tau_2\subseteq\tau_1=\tau,
$$
where $\tau_m=(\Phi^+)^{m-1}(\Phi^-)^{m-1}(\tau)$ and
$\tau_m/\tau_{m+1}=(\Phi^+)^{m-1}(\Psi^+)(\tau^{(m)})$. We 
have the following easy lemma.
\begin{lemma}\label{finite_length}
Let $\rho$ be a finite length representation of $G_t(K)$, 
where $1\leq t < n$. Then
$$ (\Phi^+)^{n-t-1}(\Psi^+)(\rho) $$ 
is also of finite length as a representation of $P_n(K)$.
\end{lemma}
\begin{proof}
This is an immediate consequence of \cite[Chapter III, Subsection
1.5]{MR1395151} and the exactness of the functor
$(\Phi^+)^{n-t-1}(\Psi^+)$.
\end{proof}
\subsubsection{}
Let $E$ be a finite Galois extension of a $p$-adic field $F$ with
$[E:F]=l$, where $l$ and $p$ are distinct primes. Fix a non-trivial
additive character $\psi_F:F\rightarrow \Lambda^\times$. By abuse 
of notation, the composition 
$$ F\xrightarrow {\psi_F} \Lambda^\times
\hookrightarrow \overline{\mathbb{Q}}_l^\times $$ 
is also denoted by $\psi_F$. Let $\psi_E$ be
the character of $E$, defined by the composition
$\psi_F\circ {\rm Tr}_{E/F}$, where ${\rm Tr}_{E/F}$ denotes the trace
map of the extension $E/F$. The mod-$l$ reductions of $\psi_F$ and
$\psi_E$ are denoted by $\overline{\psi}_F$ and $\overline{\psi}_E$,
respectively. Then, we have the following finiteness result of the Tate
cohomology groups.
\begin{proposition}\label{finiteness}
  Let $\Pi$ be a finite length $l$-modular representation of $G_n(E)$
  with an isomorphism $T:\Pi\rightarrow \Pi^\gamma$ and
  $T^l={\rm id}$. Then, the Tate cohomology $\widehat{H}^i(\Pi)$, with
  respect to the operator $T$, is a finite length representation of
  $G_n(F)$.
\end{proposition}
\begin{proof}
  We prove the proposition using induction on the integer $n$. The
  case $n=1$ is clear. So, we assume that the proposition is true for
  all finite length $l$-modular representations of
  $G_t(E)\rtimes \Gamma$ and for all $t<n$. Now, we consider $\Pi$ as
  a representation of the mirabolic subgroup $P_n(E)$. Since
  $\overline{\psi}_E(\gamma(x)) = \overline{\psi}_E(x)$, for $x\in E$,
  we get the isomorphism
\begin{equation}\label{isom_1}
\Phi^-(\Pi^\gamma)\simeq\Phi^-(\Pi)^\gamma,
\end{equation}
as representation of $P_{n-1}(E)$. Similarly, for any smooth
$l$-modular representation $\tau$ of $P_{n-1}(E)$, we have
$P_n(E)$-equivariant isomorphism
\begin{equation}\label{isom_2}
\Phi^+(\tau^\gamma) \simeq \Phi^+(\tau)^\gamma.
\end{equation}
Using (\ref{isom_1}) and (\ref{isom_2}), and the isomorphism $T$, we
get an isomorphism between the representations $\Pi_m$ and
$\Pi_m^\gamma$, and also between the representations $\Pi^{(m)}$ and
$(\Pi^{(m)})^\gamma$, for all $m\leq n$.
		
Recall that $Z_{E,m}$ denote the coset space
$P_m(E)/P_{m-1}(E)U_m(E)$. Since $P_{m-1}(E)U_m(E)$ is a
$\Gamma$-stable subgroup of $P_m(E)$, we have the following long exact
sequence of non-abelian cohomology (\cite[Appendix, Proposition
1]{MR1324577}):
\begin{equation}\label{coset_inv}
0\longrightarrow P_{m-1}(F)U_m(F)\longrightarrow
P_m(F)\longrightarrow Z_{E,m}^\Gamma\longrightarrow H^1(\Gamma,
P_{m-1}(E)U_m(E))
\longrightarrow H^1(\Gamma, P_m(E)).
\end{equation}
Consider the short exact sequence of non-abelian $\Gamma$-modules
\begin{equation}\label{claim_sequence}
0 \longrightarrow U_m(E)\longrightarrow
P_{m-1}(E)U_m(E)\longrightarrow P_{m-1}(E)\longrightarrow 0.
\end{equation}
From Hilbert's theorem $90$, we get that $H^1(\Gamma, U_m(E))$ and
$\widehat{H}^1(\Gamma,P_{m-1}(E))$ are trivial. Then, from the long
exact sequence of non-abelian cohomology corresponding to
(\ref{claim_sequence}), we have $H^1(\Gamma,
P_{m-1}(E)U_m(E))=0$. Hence, the long exact sequence (\ref{coset_inv})
gives the equality $Z_{E,m}^\Gamma=Z_{F,m}$. Now, using Proposition
\ref{tr} repeatedly $(m-1)$-times, we get the $P_n(F)$-equivariant
isomorphism
$$
\widehat{H}^i(\Pi_m/\Pi_{m+1}) \simeq
(\Phi^+)^{m-1}(\Psi^+)\big(\widehat{H}^i(\Pi^{(m)})\big).
$$
Using Leibniz formula for derivatives (\cite[Lemma 1.10, Chapter
3]{MR1395151}), we get that $\Pi^{(m)}$ is a finite length
representation of $G_{n-m}(E)$. By induction hypothesis, the
$G_{n-m}(F)$-representation $\widehat{H}^i(\Pi^{(m)})$ is of
finite length, for all $m<n$. In view of Lemma \ref{finite_length},
it follows from the above isomorphism that the $P_n(F)$-representation
$\widehat{H}^i(\Pi_m/\Pi_{m+1})$ is of finite length for all $m<n$.
		
Now, for each $m\in\{1,2,\dots,n-1\}$, we consider the short exact
sequence of $P_n(E)$-representations
$$
0\longrightarrow \Pi_{m+1}\longrightarrow \Pi_m
\longrightarrow \Pi_m/\Pi_{m+1} \longrightarrow 0.
$$ 
Since $\Gamma$ is cyclic, the corresponding long exact sequence of
Tate cohomology gives the following diagram:
$$
\begin{tikzcd}[column sep={1.5cm,between origins},
  row sep={1.932050908cm,between origins}]
  & \widehat{H}^0(\Pi_{m+1}) \arrow[rr,]  && \widehat{H}^0(\Pi_m) \arrow[rd,] &  \\
  \widehat{H}^1(\Pi_m/\Pi_{m+1}) \arrow[ru, ]&  &&  &
  \widehat{H}^0(\Pi_m/\Pi_{m+1}) \arrow[ld,] \\
  & \widehat{H}^1(\Pi_m) \arrow[lu, ] && \widehat{H}^1(\Pi_{m+1})
  \arrow[ll, ] &
\end{tikzcd}
$$
We denote the above exact sequence by $S(m)$. Now, consider the
largest integer $r$ for which $\Pi_r$ is non-zero. Using induction
hypothesis, the Tate cohomology groups $\widehat{H}^i(\Pi_r)$ is of
finite length as a representation of $P_n(F)$.  Now, using the exact
sequence $S(r-1)$ and the finiteness of
$\widehat{H}^i(\Pi_{r-1}/\Pi_r)$, we get that
$\widehat{H}^i(\Pi_{r-1})$ is a finite length representation of
$P_n(F)$. Again, using the finiteness of both the representations
$\widehat{H}^i(\Pi_{r-1})$ and $\widehat{H}^i(\Pi_{r-2}/\Pi_{r-1})$,
it follows from the exact sequence $S(r-2)$ that
$\widehat{H}^i(\Pi_{r-2})$ is of finite length. Thus, inductively, we
 get that $\widehat{H}^i(\Pi)$ is of finite length as a
representation of $P_n(F)$ and hence of $G_n(F)$. This completes the
proof.
\end{proof}
As a corollary, we have
\begin{corollary}\label{finiteness_integral}
Let $(\Pi,V)$ be an integral $\mathcal{K}$-representation of
$G_n(E)$ with an isomorphism $T:(\Pi,V) \rightarrow (\Pi^\gamma,V)$
and $T^l = {\rm id}$. Then, for any $G_n(E)\rtimes \Gamma$-invariant
$\Lambda$-lattice $\mathcal{L}$ in $V$ (Here, $\Gamma$ acts on $V$
via $T$), the Tate cohomology groups $\widehat{H}^i(\mathcal{L})$,
$i\in\{0,1\}$, are of finite length as representations of $G_n(F)$.
\end{corollary}
\begin{proof}
Recall that $\mathcal{L}$ is a free $\Lambda$-module and 
$\mathcal{L}/l\mathcal{L}$ is a finite length $G_n(E)$-representation
(see \cite[II.5.11.a]{MR1395151} for finiteness).  Consider the short
exact sequence of $G_n(E)\rtimes\Gamma$-modules
$$ 0\longrightarrow \mathcal{L}\xrightarrow{{\rm mult.}l}
\mathcal{L}\longrightarrow \mathcal{L}/l\mathcal{L}
\longrightarrow 0. $$ 
By Proposition \ref{finiteness}, the Tate cohomology group
$\widehat{H}^i(\mathcal{L}/l\mathcal{L})$ has finite length as
representations of $G_n(F)$. From the long exact
sequence of Tate cohomology corresponding to the above short exact
sequence, we have
$$ 0\rightarrow \widehat{H}^0(\mathcal{L})\rightarrow
\widehat{H}^0(\mathcal{L}/l\mathcal{L})\rightarrow
\widehat{H}^1(\mathcal{L})\rightarrow 0. $$
and 
$$ 0\rightarrow \widehat{H}^1(\mathcal{L})\rightarrow
\widehat{H}^1(\mathcal{L}/l\mathcal{L})\rightarrow
\widehat{H}^0(\mathcal{L})\rightarrow 0. $$
Thus, we get that each
$\widehat{H}^i(\mathcal{L})$ is of finite length.
\end{proof}
	
\subsection{Frobenius Twist}\label{frob_twist}
Let $G$ be a locally profinite group (i.e., locally compact and
totally disconnected). Let $(\sigma, V)$ be an $l$-modular
representation of $G$. Consider the vector space $V^{(l)}$, where the
underlying additive group structure of $V^{(l)}$ is same as that of
$V$ but the scalar action $*$ on $V^{(l)}$ is given by
\begin{center}
$c * v = c^{\frac{1}{l}}v$, for all $c \in \overline{\mathbb{F}}_l$, $v\in V$.
\end{center}
Then the action of $G$ on $V$ induces a representation $\sigma^{(l)}$
of $G$ on $V^{(l)}$. The representation $(\sigma^{(l)}, V^{(l)})$ is
called the Frobenius twist of the representation $(\sigma,V)$.
	
We end this subsection with a lemma which will be used in the main result.
\begin{lemma}\label{ar}
Let $\psi$ be a non-trivial $l$-modular additive character of $F$
and let $\Theta$ be the non-degenerate character of $N_n(F)$
corresponding to $\psi$. If $(\pi,V_\pi)$ and 
$(\sigma, V_\sigma)$
are two $l$-modular generic representations of $G_n(F)$
and $G_{n-1}(F)$ respectively, then
$$ \gamma(X,\pi,\sigma,\psi)^l=\gamma(X^l,\pi^{(l)},
\sigma^{(l)},\psi^l). $$
\end{lemma}
\begin{proof}
Let $W_\pi$ be a Whittaker functional on the representation $\pi$.
Then the composite map
$$ V_\pi \xrightarrow{W_\pi} \overline{\mathbb{F}}_l
\xrightarrow{x\mapsto x^l} \overline{\mathbb{F}}_l, $$ 
denoted by $W_{\pi^{(l)}}$, is a Whittaker functional (with respect to
$\psi^l:N_n(F)\rightarrow \overline{\mathbb{F}}_l^\times$) 
on the representation $\pi^{(l)}$, as we have:
$$ W_{\pi^{(l)}}(c.v) = W_{\pi}((c^{\frac{1}{l}}v))^l=
c W_{\pi^{(l)}}(v) $$
and
$$ W_{\pi^{(l)}}(\pi^{(l)}(n)v) = (\Theta(n)W_{\pi}(v))^l
= \Theta^l(n)W_{\pi^{(l)}}(v),$$ for all $v \in V_{\pi}$,
$c \in \overline{\mathbb{F}}_l$ and all $n \in N_n(F)$.
		
So the Whittaker model $\mathbb{W}(\pi^{(l)},\psi^l)$ 
consists of the
functions $W_v^l$, where $W_v$ varies in
$\mathbb{W}(\pi,\psi)$. Similarly the Whittaker model
$\mathbb{W}(\sigma^{(l)}, \psi^l)$ of $\sigma^{(l)}$ consists of the
functions $U_v^l$, where $U_v$ varies in
$\mathbb{W}(\sigma,\psi)$. Then by the Rankin-Selberg functional
equation in the subsection (\ref{w}), we have
\begin{equation}\label{30}
\sum_{r\in\mathbb{Z}}
c^F_r(\widetilde{W_v},\widetilde{U_v})^lq_F^{-lr/2}X^{-lr} =
\omega_\sigma(-1)^{n-2}\gamma(X,\pi,\sigma,\psi)^l
\sum_{r\in\mathbb{Z}}c^F_r(W_v,U_v)^lq_F^{lr/2}X^{lr}
\end{equation}
and 
\begin{equation}\label{31}
\sum_{r\in\mathbb{Z}}
c^F_r(\widetilde{W_v^l},\widetilde{U_v^l})q_F^{-r/2}X^{-r} =
\omega_{\sigma^{(l)}}(-1)^{n-2}\gamma(X,\pi^{(l)},\sigma^{(l)},\psi^l)
\sum_{r \in \mathbb{Z}}c^F_r(W_v^l,U_v^l)q_F^{r/2}X^r.
\end{equation}
Replace $X$ by $X^l$ to the equation (\ref{31}), we have
\begin{equation}\label{32}
\sum_{r\in\mathbb{Z}}
c^F_r(\widetilde{W_v^l},\widetilde{U_v^l})q_F^{-r/2}X^{-lr} =
\omega_{\sigma^{(l)}}(-1)^{n-2}\gamma(X^l,\pi^{(l)},\sigma^{(l)},\psi^l)
\sum_{r \in \mathbb{Z}}c^F_r(W_v^l,U_v^l)q_F^{r/2}X^{lr}.
\end{equation}
Then from the equations (\ref{30}) and (\ref{32}), we get
$$ \gamma(X,\pi,\sigma,\psi)^l = \gamma(X^l,\pi^{(l)},\sigma^{(l)},\psi^l). $$
\end{proof}
	
\section{Tate Cohomology of Whittaker Lattice}\label{Tate_main_result}
As before, we fix a non-trivial additive character
$\psi_F:F\rightarrow \Lambda^\times$. Let $\psi_E$ be
the composition $\psi_F \circ {\rm Tr}_{E/F}$, where ${\rm Tr}_{E/F}$
is the trace map of the extension $E/F$. The mod-$l$ reductions of
$\psi_F$ and $\psi_E$ are denoted by $\overline{\psi}_F$ and
$\overline{\psi}_E$ respectively. Let $\Theta_F$ and $\Theta_E$ be the
characters of $N_n(F)$ and $N_n(E)$ respectively, as defined in
(\ref{rr}). Similarly, we denote by $\overline{\Theta}_F$ and
$\overline{\Theta}_E$ the mod-$l$ reductions of $\Theta_F$ and
$\Theta_E$ respectively.
\subsection{}
Let $(\pi, V)$ be a generic $R$-representation of $G_n(E)$, where
$R=\overline{\mathbb{Q}}_l$ or $\overline{\mathbb{F}}_l$, such that
$\pi$ is isomorphic to $\pi^\gamma$, for all $\gamma\in \Gamma$. Let
$\mathbb{W}(\pi, \psi_E)$ be the Whittaker model of $\pi$. For
$W\in \mathbb{W}(\pi, \psi_E)$, we recall that $\gamma.W$ is a
function given by
$$ (\gamma.W)(g) = W(\gamma^{-1}(g)), $$
for all $g\in G_n(E)$. Note that $\gamma.W\in \mathbb{W}(\pi,\psi_E)$ 
(see Lemma \ref{inv_whit_model}). Thus, we define
$$ T_\gamma: \mathbb{W}(\pi,\psi_E)\rightarrow 
\mathbb{W}(\pi, \psi_E) $$ 
by setting $T_\gamma(W)=\gamma.W$, for all 
$W\in \mathbb{W}(\pi, \psi_E)$. 
The map $T_\gamma$ gives an isomorphism between 
$(\pi^\gamma, V)$ and $(\pi, V)$ as we have
$$ T_\gamma(\pi(g)W)(h)=\pi(g)W(\gamma^{-1}(h)) =
W(\gamma^{-1}(h)g) $$
and
$$ [\pi^\gamma(g)T_\gamma(W)](h) = T_\gamma(W)(h\gamma(g)) =
W(\gamma^{-1}(h)g), $$
for all $g, h\in G_n(E)$. 
	
\subsection{Jacquet-functors and Tate cohomology}
We begin with a few elementary results on the compatibility of Jacquet
(twisted Jacquet) functors with Tate cohomology. Let
$(\pi, \mathcal{L})$ be a smooth
$\Lambda[G_n(E)\rtimes \Gamma]$-module. Let
$\lambda=(n_1,n_2,\dots,n_r)$ be a partition of $n$ and let
$P_\lambda = M_\lambda N_\lambda$ be a parabolic subgroup of $G_n$
with $N_\lambda$ its unipotent radical and $M_\lambda$ is a standard
Levi subgroup. Let $\mathcal{L}(N_\lambda(E))$ be the space spanned by
the set of vectors
$$ \big\{v-\pi(n)v: v\in \mathcal{L}, n\in N_\lambda(E)\big\}. $$
Note that the space $\mathcal{L}(N_\lambda(E))$ is stable under the
action of $\Gamma$.
\begin{lemma}\label{jacquet-langlands-tate}
The image of the natural map $\widehat{H}^0(\mathcal{L}
(N_\lambda(E))) \rightarrow \widehat{H}^0(\mathcal{L})$ 
is equal to $\widehat{H}^0(\mathcal{L})(N_\lambda(F))$.
\end{lemma}
\begin{proof}
Let $\phi$ be the natural map
$\widehat{H}^0(\mathcal{L}(N_\lambda(E)))\rightarrow
\widehat{H}^0(\mathcal{L})$. Let $v\in \text{img}(\phi)$ and let
$\tilde{v}$ be a lift of $v$ in
$\mathcal{L}(N_\lambda(E))^\Gamma$. Then there exists a compact open
subgroup $\mathcal{N}$ of $N_\lambda(E)$ such that
\begin{equation}\label{lift_sum_1}
\int_{\mathcal{N}}\pi(n)\tilde{v}\,dn = 0.
\end{equation}
Since $\pi$ is smooth, there exists a compact open subgroup
$\mathcal{N}'$ of $\mathcal{N}$ of finite index such that
$$ \int_{\mathcal{N}}\pi(n)\tilde{v}\,dn =
\sum_{n\in \mathcal{N}/\mathcal{N}'}\pi(n)\tilde{v}\,dn. $$
Since $N_\lambda(E)$ has a filtration of $\Gamma$-stable compact open
subgroups, we may assume that $\mathcal{N}$ is $\Gamma$-stable. If $X$
denotes the coset space $\mathcal{N}/\mathcal{N}'$, then we have
\begin{equation}\label{sum_lift_2}
\sum_{x\in X} \pi(x)\tilde{v} =  \sum_{y\in X^\Gamma} 
\pi(y)\tilde{v}\ + \sum_{z\in {X\setminus X^\Gamma}} \pi(z)\tilde{v}.
\end{equation}
Since the $\Gamma$-action on $X\setminus X^\Gamma$ is free, there
exists a subset $U$ such that $X\setminus X^\Gamma$ is the disjoint
union of $\gamma^iU$, $1\leq i \leq l$. As $\tilde{v}$ is
$\Gamma$-invariant, we have
$$
\sum_{z\in {X\setminus X^\Gamma}} \pi(z)\tilde{v} =
\sum_{i=1}^l\sum_{u\in U}\pi(\gamma^i(u))\tilde{v} = 
N\big(\sum_{u\in U} \pi(u)\tilde{v}\big),
$$
where $N = 1+\gamma+\cdots +\gamma^{l-1}$. This shows that 
$$ \sum_{z\in {X\setminus X^\Gamma}} \pi(z)v = 0 $$ 
in $\widehat{H}^0(\mathcal{L})$.  Therefore, it follows from
(\ref{lift_sum_1}) and (\ref{sum_lift_2}) that
$$ \sum_{y\in X^\Gamma} \pi(y)v = 0. $$
This implies that $v$ belongs to
$\widehat{H}^0(\mathcal{L})(N_\lambda(F))$. Conversely, let $w$ be 
an element of $\widehat{H}^0(\mathcal{L})(N_\lambda(F))$. Let
$\mathcal{N}_F$ be the compact open subgroup of $N_\lambda(F)$ with
$$ \int_{\mathcal{N}_F}\pi(n)w\,dn = 0. $$
Choose a compact open subgroup $\mathcal{N}_F'$ of $\mathcal{N}_F$ of
finite index such that
$$ \int_{\mathcal{N}_F}\pi(n)w\,dn =
\sum_{n\in \mathcal{N}_F/\mathcal{N}_F'}\pi(n)w\,dn. $$
Let $\tilde{w}$ be a lift of $w$ in $\mathcal{L}^\Gamma$. Consider 
the element
$$ \tilde{w}_1 = \tilde{w} -
\frac{1}{\lvert\mathcal{N}_F/\mathcal{N}_F'\rvert} \sum_{n\in
\mathcal{N}_F/\mathcal{N}_F'}\pi(n)\tilde{w}.
$$
Then $\tilde{w}_1$ belongs to $\mathcal{L}(N_\lambda(E))^\Gamma$ 
and $\phi(\tilde{w}_1) = w$. Moreover, we have
$$ \int_{\mathcal{N}_F} \pi(n)\tilde{w}_1\,dn =
\sum_{n\in \mathcal{N}_F/\mathcal{N}_F'} \pi(n)\tilde{w}_1 = 0. $$
This completes the proof.
\end{proof}
\begin{lemma}\label{Jaquet_Tate}
Let $(\pi, V)$ be a smooth $l$-modular representation of 
$G_n(E)\rtimes \Gamma$ such that
$\widehat{H}^1(V_{N_\lambda(E)}) = 0$,
$\widehat{H}^0(V)_{N_\lambda(F)}$ is non-zero and
$\widehat{H}^0(V_{N_\lambda(E)})$ is irreducible 
under $M_\lambda(F)$. Then, the $M_\lambda(F)$-representation
$\widehat{H}^0(V)_{N_\lambda(F)}$ is isomorphic to
$\widehat{H}^0(V_{N_\lambda(E)})$.
\end{lemma}
\begin{proof}
The long exact sequence of Tate cohomology groups associated  
with the exact sequence 
$$ 0 \longrightarrow V(N_\lambda(E))\longrightarrow V
\longrightarrow V_{N_\lambda(E)}\longrightarrow 0, $$
is equal to :
$$ 
0\rightarrow \widehat{H}^0(V(N_\lambda(E)))
\xrightarrow{\phi}\widehat{H}^0(V)\rightarrow
\widehat{H}^0(V_{N_\lambda(E)}) \rightarrow
\widehat{H}^1(V(N_\lambda(E)))\rightarrow
\widehat{H}^1(V)\rightarrow 0.
$$ 
Using Lemma \ref{jacquet-langlands-tate}, we get that
$\phi\big(\widehat{H}^0(V(N_\lambda(E)))\big)$ is equal to
$\widehat{H}^0(V)(N_\lambda(F))$, and therefore the
$N_\lambda(F)$-coinvariants $\widehat{H}^0(V)_{N_\lambda(F)}$ is
isomorphic to a subrepresentation of
$\widehat{H}^0(V_{N_\lambda(E)})$. Since
$\widehat{H}^0(V)_{N_\lambda(F)}$ is non-zero, the lemma follows from
the irreducibility of $\widehat{H}^0(V_{N_\lambda(E)})$.
\end{proof}
Using similar ideas, we can prove that zeroth Tate cohomology of a
generic representation has a unique generic subquotient. For any
integral $l$-adic generic representation $(\pi, V)$ of $G_n(K)$,
defined over $\mathcal{K}$, we recall that
$\mathbb{W}_\Lambda(\pi, \psi_K)$ denotes the set of functions in
$\mathbb{W}(\pi_E, \psi_E)$ with values in $\Lambda$. The
$\Lambda$-module $\mathbb{W}_\Lambda(\pi, \psi_K)$ gives a $\Lambda$
structure of $\pi$ (see \ref{vigneras_whittaker_lattice}).
	
\begin{proposition}\label{tate_is_generic}\label{Tate_integral_generic}
  Let $\pi_E$ be an $l$-modular generic representation (or an integral
  $l$-adic generic representation defined over $\mathcal{K}$) of
  $G_n(E)$. Assume that $\pi_E$ is stable under the action of
  $\Gamma$. Then there exists a unique generic subquotient of the
  $G_n(F)$ representation $\widehat{H}^0(\mathbb{W}(\pi_E, \psi_E))$
  (resp.  $\widehat{H}^0(\mathbb{W}_\Lambda(\pi_E, \psi_E))$).
\end{proposition}
\begin{proof}
Let $\mathcal{W}$ be the Whittaker space
$\mathbb{W}(\pi_E, \overline{\psi}_E)$
( resp. $\mathbb{W}_\Lambda(\pi_E, \psi_E)$). Let
$\mathcal{W}(N_n(E), \Theta_E)$ be the $\overline{\mathbb{F}}_l$
(resp. $\Lambda)$-span of the vectors of the form
$\pi_E(n)v-\overline{\Theta}_E(n)v$
(resp. $\pi_E(n)v-\Theta_E(n)v$), for all $v\in \mathcal{W}$ and
$n\in N_n(E)$. We have the following exact sequence :
$$
0\rightarrow \mathcal{W}(N_n(E),
\Theta_E)\rightarrow 
\mathcal{W}\rightarrow \mathcal{W}_{N_n(E),
\Theta_E}\rightarrow 0.
$$ 
The space $\mathcal{W}_{N_n(E), \Theta_E}$ is a one dimensional vector
space over $\overline{\mathbb{F}}_l$ (resp. a free $\Lambda$-module of
rank one). The long exact sequence in the Tate cohomology gives us
$$
\widehat{H}^0(\mathcal{W}(N_n(E),\Theta_E))\xrightarrow{f}
\widehat{H}^0(\mathcal{W})\xrightarrow{g}
\widehat{H}^0(\mathcal{W}_{N_n(E),\Theta_E})\rightarrow
\widehat{H}^1(\mathcal{W}(N_n(E),\Theta_E))\rightarrow
\widehat{H}^1(\mathcal{W}).
$$ 
Using arguments of Lemma \ref{jacquet-langlands-tate}, the image of
the morphism $f$ is equal to
$\widehat{H}^0(\mathcal{W})(N_n(F),\Theta_F^l)$. The Tate cohomology
of the Kirillov model $\mathbb{K}(\pi_E,\psi_E)$ contains
$\mathcal{K}(\psi_F^l)$ as $P_n(F)$ subrepresentation (see
\ref{main_proof_first_part}). Since $\Gamma$ acts trivially on
$\mathcal{W}_{N_n(E), \Theta_E}$, the Tate cohomology space
$\widehat{H}^0(\mathcal{W}_{N_n(E),\Theta_E})$ is a one dimensional
vector space over $\overline{\mathbb{F}}_l$. Hence, the map $g$
induces the isomorphism:
$$
\widehat{H}^0(\mathcal{W})_{N_n(F),\Theta_F^l}\simeq 
\widehat{H}^0(\mathcal{W}_{N_n(E),\Theta_E}).
$$ 
Now, it follows from the exactness of the Jacquet functor and
Proposition \ref{finiteness}, that
$\widehat{H}^0(\mathcal{W})$ admits a unique generic subquotient.
\end{proof}
\begin{remark}\normalfont
  The above lemmas will be used to compute the Tate cohomology of the
  base change of mod-$l$ the Zelevinsky sub-representation
  $Z(\Delta)$. The Jacquet functor of $Z(\Delta)$ with respect to the
  parabolic subgroup of type $(n/k, n/k,\dots, n/k)$, where $k$ is the
  length of the segment $\Delta$, is an $l$-modular cuspidal
  representation and the hypothesis in Lemma \ref{Jaquet_Tate} are
  applicable. The precise definitions will be recalled in the next
  section.
\end{remark}
\subsection{The \texorpdfstring{${\rm GL}_2$}{} case}
\begin{theorem}\label{n=2_thm}
Let $F$ be a finite extension of $\mathbb{Q}_p$, and let $E$ be a
finite Galois extension of $F$ with $[E:F] =l$. Assume that $l$ and
$p$ are distinct odd primes. Let $\pi_F$ be an integral $l$-adic
cuspidal representations of $G_2(F)$ and let $\pi_E$ be the
representation of $G_2(E)$ such that $\pi_E$ is the base change of
$\pi_F$. Then
\begin{center}
$\widehat{H}^0(r_l(\pi_E))\simeq r_l(\pi_F)^{(l)}$.
\end{center}
\end{theorem}
\begin{proof}
First note that the base change lift $\pi_E$ is cuspidal, and hence
the mod-$l$ reduction $r_l(\pi_E)$ is also cuspidal. Let $\psi_E$
and $\psi_F$ be defined as in subsection (\ref{sr}). Let
$\big(\mathbb{K}_{\overline{\psi}_E}^{r_l(\pi_E)},
C_c^\infty(E^\times,\mathbb{\overline{F}}_l)\big)$ be a Kirillov
model of $r_l(\pi_E)$. Recall that the group $\Gamma$ acts on
$C_c^\infty(E^\times,\mathbb{\overline{F}}_l)$. We denote by
$\widehat{H}^0(r_l(\pi_E))$ the cohomology group
$\widehat{H}^0\big(C_c^\infty(E^\times,
\mathbb{\overline{F}}_l)\big)$. Then, using Proposition \ref{tr}, 
we have
\begin{center}
$\widehat{H}^0(r_l(\pi_E)) \simeq
C_c^\infty(F^\times,\mathbb{\overline{F}}_l)$.
\end{center}
The space $\widehat{H}^0(r_l(\pi_E))$ is isomorphic to
${\rm ind}_{N_2(F)}^{P_2(F)}(\overline{\psi}_F^l)$ as a representation
of $P_2(F)$, where $\overline{\psi}_F$ is the mod-$l$ reduction of
$\psi_F$; and the induced action of the operator
$\mathbb{K}_{\overline{\psi}_E}^{r_l(\pi_E)}(w)$ on
$\widehat{H}^0(r_l(\pi_E))$ is denoted by
$\overline{\mathbb{K}_{\overline{\psi}_E}^{r_l(\pi_E)}(w)}$. The
theorem now follows from the following claim.
\begin{claim}\normalfont
$\overline{\mathbb{K}_{\overline{\psi}_E}^{r_l(\pi_E)}(w)}(f)
= \mathbb{K}_{\overline{\psi}_F^l}^{r_l(\pi_F)^{(l)}}(w)(f)$, 
for all $f\in C_c^\infty(F^\times,\mathbb{\overline{F}}_l)$.
\end{claim}
Now, for a function
$f\in C_c^\infty(F^\times,\mathbb{\overline{F}}_l)$, any covering of
${\rm supp}(f)$ by open subsets of $F^\times$ has a finite refinement
of pairwise disjoint open compact subgroups of $F^\times$. So we may
assume that ${\rm supp}(f) \subseteq \varpi^rxU_F^1$, where
$r\in \mathbb{Z}$, $\varpi_F$ is an uniformizer of $F$ and $x$ is a
unit in $(\mathfrak{o}_F/\mathfrak{p}_F)^\times$ embedded in
$F^\times$. Therefore it is sufficient to prove the claim for
functions $f\in C_c^\infty(F^\times,\mathbb{\overline{F}}_l)$ with
${\rm supp}(f) \subseteq U_F^1$, and we have
$$
f = c_{\chi_F}\sum_{\chi_F\in\widehat{U_F^1}}\xi\{\chi_F,0\},
$$
where $c_{\chi_F}\in\overline{\mathbb{F}}_l$ and
$\widehat{U_F^1}$ is the set of smooth characters of
$$ F^\times=\langle\varpi_F\rangle\times k_F^\times\times U_F^1 $$ 
which are trivial on $k_F^\times$ and
$\varpi_F$. We now prove the claim for the function
$\xi\{\chi_F,0\}$ for
$\chi_F\in\widehat{U_F^1}$. There exists a character $\chi_0\in
\widehat{U_F^1}$ such that $\chi_0^l=\chi_F$. Let
$\widetilde{\chi}_0$ be the $l$-adic lift of the character
$\chi_0$. Define the characters $\chi_E$ and $\widetilde{\chi}_E$ 
of $E^\times$ as follows
$$ \chi_E(x) = \chi_0({\rm Nr}_{E/F}(x)) $$
and
$$ \widetilde{\chi}_E(x) = \widetilde{\chi}_0
({\rm Nr}_{E/F}(x)), $$
for $x \in E^\times$. Here, ${\rm Nr}_{E/F}:E^\times\rightarrow
F^\times$ denotes the norm map. Note that
$\widetilde{\chi}_E$ and $\chi_E$ extends the character
$\chi_F$. We have the following relations :
\begin{equation}\label{kirillov_equ_1}
\overline{\mathbb{K}_{\overline{\psi}_E}^{r_l(\pi_E)}(w)}
\big(\xi\{\chi_F,0\}\big) =
\epsilon(\chi_E^{-1}r_l(\pi_E),\overline{\psi}_E)
\xi\bigg\{\chi_F,\cfrac{-n(\chi_E^{-1}
r_l(\pi_E),\overline{\psi}_E)}{e}\bigg\}
\end{equation}
and
\begin{equation}\label{kirillov_equ_2}
\mathbb{K}_{\overline{\psi}_F^l}^{r_l(\pi_F)^{(l)}}(w)
\big(\xi\{\chi_F,0\}\big) = \epsilon(\chi_F^{-1}r_l(\pi_F)^{(l)},\overline{\psi}_F^l)
\xi\big\{\chi_F,-n(\chi_F^{-1}r_l(\pi_F)^{(l)},
\overline{\psi}_F^l)\big\},
\end{equation}
where $e$ denotes the ramification index of the extension
$E/F$. Next, we aim to prove the following identity:
$$ \epsilon(\chi_E^{-1}r_l(\pi_E),\overline{\psi}_E) =
\epsilon(X,\chi_F^{-1}r_l(\pi_F)^{(l)},\overline{\psi}_F^l). $$ It
follows from Theorem \ref{k} that the $\epsilon$-factor is same as the
$\gamma$-factor in both $l$-adic and mod-$l$ cases.  Now, using the
identity in \cite[Proposition 6.9]{MR1007299}, we get
\begin{equation}\label{base_change_gamma_identity}
\epsilon(X,\widetilde{\chi}_E^{-1}\pi_E,\psi_E) = 
\prod_{\eta} \epsilon(X,\widetilde{\chi}_0^{-1}\pi_F
\otimes\eta,\psi_F),
\end{equation} 
where $\eta$ runs over all the characters of the group
$F^\times/{\rm Nr}_{E/F}(E^\times)$-- which is isomorphic to
${\rm Gal}(E/F)$ via local class field theory. Using the identity
(\ref{gamma_factor_generic_part}), we have
$$ r_l\big(\epsilon(X,\widetilde{\chi}_E^{-1}\pi_E,\psi_E)\big) = 
\epsilon(X,\chi_E^{-1}r_l(\pi_E),\overline{\psi}_E) $$
and
$$ r_l\big(\epsilon(X,\widetilde{\chi}_0^{-1}\pi_F
\otimes\eta,\psi_F)\big) =
\epsilon(X,\chi_0^{-1}r_l(\pi_F),\overline{\psi}_F), $$ for each
character $\eta$. Using the above identities and taking mod-$l$
reduction to (\ref{base_change_gamma_identity}), we get
$$ \epsilon(X,\chi_E^{-1}r_l(\pi_E),\overline{\psi}_E) = 
\epsilon(X,\chi_0^{-1}r_l(\pi_F),\overline{\psi}_F)^l. $$ 
Using Lemma \ref{ar}, we obtain the following identity
$$ \epsilon(X,\chi_E^{-1}r_l(\pi_E),\overline{\psi}_E) =
\epsilon(X^l,\chi_F^{-1} r_l(\pi_F)^{(l)},\overline{\psi}_F^l). 
$$
Now, using the identity (\ref{degree}) and comparing
the degree of $X$ from above relation, we get
$$
\cfrac{n(\chi_E^{-1}r_l(\pi_E), \overline{\psi}_E)}{e} =
n(\chi_F^{-1}r_l(\pi_F)^{(l)},
\overline{\psi}_F^l)
$$
and 
$$
\epsilon(\chi_E^{-1}r_l(\pi_E),\overline{\psi}_F) =
\epsilon(\chi_F^{-1} r_l(\pi_F)^{(l)},\overline{\psi}_F^l).
$$
Thus it follows from (\ref{kirillov_equ_1}) and 
(\ref{kirillov_equ_2}) that
$$
\overline{\mathbb{K}_{\overline{\psi}_E}^{r_l(\pi_E)}(w)}
\big(\xi\{\chi_F,0\}\big) = \mathbb{K}_{\overline{\psi}_F^l}^{r_l(\pi_F)^{(l)}}(w)
\big(\xi\{\chi_F,0\}\big).
$$
Hence we prove the claim, and the theorem follows.
\end{proof}
\subsection{}
Our main result uses the following lemma which is the analogue
of completeness of Whittaker models in the complex case.
\begin{lemma}\label{van}
Assume that $l$ does not divide
$|G_n(k_K)|$ and let $\overline{\psi}_K$ be the
mod-$l$ reduction of $\psi_K$. Let
$\overline{\Theta}_K$ be the non-degenerate character of
$N_n(K)$ associated with
$\overline{\psi}_K$ (see Section \ref{rr}). Let $\phi\in{\rm
ind}_{N_n(K)}^{G_n(K)}(\overline{\Theta}_K)$. If
$$ \int_{N_n(K)\setminus G_n(K)}\phi(t)W(t)\,dt=0, $$
for all $W \in \mathbb{W}(\sigma,
\overline{\psi}_K^{-1})$ and for all generic representations
$\sigma$ of $G_n(K)$, then $\phi = 0$.
\end{lemma}
\begin{proof}
  Suppose $\phi$ is non-zero. Let
  ${\rm Rep}_{W(\overline{\mathbb{F}}_l)}(G_n(K))$ be the category of
  smooth $W(\overline{\mathbb{F}}_l)[G_n(K)]$-modules, and let
  $\mathcal{Z}_n$ be its center. Let $W_n$ be the smooth
  $W(\overline{\mathbb{F}}_l)[G_n(K)]$-module
  ${\rm ind}_{N_n(K)}^{G_n(K)}(\Theta_K)$. Recall that for any
  primitive idempotent $e$ in $\mathcal{Z}_n$, the space $eW_n$ is a
  smooth co-Whittaker $e\mathcal{Z}_n[G_n(K)]$-module (see
  \cite[Theorem 6.3]{Helm_integral_bernstein_center}). According to
  \cite[Corollary 4.3]{MR4335904}, there exists a primitive idempotent
  $e'$ of $\mathcal{Z}_n$ and an element
  $U\in \mathbb{W}(e'W_n\otimes_{W(\overline{\mathbb{F}}_l)}
  \overline{\mathbb{F}}_l, \overline{\psi}_K^{-1})$ such that the
  integral
$$
\langle \phi, U\rangle := \int_{N_n(K)\setminus G_n(K)}\phi(t)
\otimes U(t)\,dt
$$
is non-zero in
$\overline{\mathbb{F}}_l\otimes_{W(\overline{\mathbb{F}}_l)}
e'\mathcal{Z}_n$. As described in \cite{MR3508741}, the primitive
idempotent $e'$ corresponds to an inertial equivalence class of pairs
$(M,\pi)$, where $M$ is a Levi subgroup of $G_n(K)$ and $\pi$ is a
supercuspidal $\overline{\mathbb{F}}_l$-representation of $M$. Let
$R'$ denote the ring
$\overline{\mathbb{F}}_l\otimes_{W(\overline{\mathbb{F}}_l)}
e'\mathcal{Z}_n$.
		
For the inertial equivalence class $[M,\pi]$, consider the subcategory
${\rm Rep}_{W(\overline{\mathbb{F}}_l)}(G_n(K))_{[M,\pi]}$, consisting
of objects $\Pi$ in ${\rm Rep}_{W(\overline{\mathbb{F}}_l)}(G_n(K))$
whose irreducible sub-quotients have mod-$l$ inertial supercuspidal
support $[M,\pi]$. Let $A_{[M,\pi]}$ denote the center of the
subcategory
${\rm Rep}_{W(\overline{\mathbb{F}}_l)}(G_n(K))_{[M,\pi]}$. Since $l$
does not divide $\lvert G_n(k_K)\rvert$, it follows from \cite[Example
13.9]{MR3508741} that
$$ A_{[M,\pi]} = C_{[M,\pi]}, $$ 
where $C_{[M,\pi]}$ is a $W(\overline{\mathbb{F}}_l)$-subalgebra of
$A_{[M,\pi]}$, as defined in \cite[Theorem 12.5]{MR3508741}. Then,
there is an isomorphism of
$C_{[M,\pi]}\otimes_{W(\overline{\mathbb{F}}_l)}
\overline{\mathbb{F}}_l$ with the reduced quotient of
$A_{[M,\pi]}\otimes_{\Lambda} \overline{\mathbb{F}}_l$ (see
\cite[Corollary 12.13]{MR3508741}), and hence we get that the
$W(\overline{\mathbb{F}}_l)$-algebra $R'$ is reduced. In particular,
the element $\langle \phi,U\rangle$ is not nilpotent. Therefore, the
basic open set $D(\langle \phi,U\rangle)$ is non-empty, and hence
intersects the dense set of closed points of the affine
$W(\overline{\mathbb{F}}_l)$-scheme associated with $R'$. This implies
that there exists a map $f : R'\rightarrow \overline{\mathbb{F}}_l$
such that the image of $\langle \phi, U\rangle$ under $f$, which is
equal to
$$
\int_{N_n(K)\setminus G_n(K)}\phi(t)W_0(t)\,dt
$$ 
for some
$W_0\in\mathbb{W}\big(e'W_n \otimes_{R', f}
\overline{\mathbb{F}}_l,\overline{\psi}_K^{-1}\big)$, is non-zero in
$\overline{\mathbb{F}}_l$. Note that
$e'W_n\otimes_{R', f}\overline{\mathbb{F}}_l$, as
$\overline{\mathbb{F}}_l$-representation, admits a generic quotient
with same Whittaker space. Hence, the lemma.
\end{proof}
	
\subsection{The general case}
Let $\pi_F$ be an integral generic $l$-adic representation of
$G_n(F)$, and let $\pi_E$ be the base change lifting of $\pi_F$ to
$G_n(E)$. We observe that the unique generic component $J_l(\pi_E)$ of
the mod-$l$ reduction of $\pi_E$ is stable under the action of
$\Gamma$. We will now prove the main theorem of our article.
\begin{theorem}\label{noncuspidal_n_thm}
Let $F$ be a finite extension of $\mathbb{Q}_p$, and let $E$ be a
finite Galois extension of $F$ with $[E:F] =l$, where $p$ and $l$
are distinct primes such that $l$ does not divide
$|G_{n-1}(k_F)|$. Let $\pi_F$ be an integral $l$-adic generic
representation of $G_n(F)$ with $J_l(\pi_F)$, the unique generic
component of the mod-$l$ reduction of $\pi_F$. Let $\pi_E$ be the
base change lift of $\pi_F$. Then, the representation
$J_l(\pi_F)^{(l)}$ is the unique generic sub-quotient of
$\widehat{H}^0(J_l(\pi_E))$.
\end{theorem}
\begin{proof}
  We begin with a summary of the proof. We prove the above theorem
  using induction on the integer $n$. The proof is divided into four
  parts. In the first part, we isolate a subspace
  $\mathcal{M}(\pi_F, \psi_F)$ of the Tate cohomology of the Kirillov
  model of $J_l(\pi_E)$ which will eventually give $J_l(\pi_F)^{(l)}$
  as a quotient. In the second part, we will set up comparison of Zeta
  integrals on homogeneous spaces of $F$ with those on homogeneous
  spaces of $E$. In the third part we reduce the theorem to an
  identity of local $\gamma$-factors. In the fourth-part we deal with
  these local $\gamma$-factor identities and we show that
  $\mathcal{M}(\pi_F, \psi_F)$ is stable under the action of
  $G_n(F)$. At the end of the fourth part, we get a natural onto map
  from $\mathcal{M}(\pi_F, \psi_F)$ to the mod-$l$ Kirillov model
  $\mathbb{K}(J_l(\pi_F)^{(l)}, \overline{\psi}_F^l)$ as $G_n(F)$
  representations.
\subsubsection{}\label{main_proof_first_part}
Notations on Whittaker and Kirillov models are defined in the
subsections (\ref{whittaker_recap}) and
(\ref{kirrilov_model}). Consider the Whittaker model
$\mathbb{W}(J_l(\pi_E), \overline{\psi}_E)$ of $J_l(\pi_E)$. The
restriction map $W\mapsto{\rm res}_{P_n(E)}(W)$ is an isomorphism from
$\mathbb{W}(J_l(\pi_E), \overline{\psi}_E)$ onto
$\mathbb{K}(J_l(\pi_E),\overline{\psi}_E)$ (see \cite[Theorem
4.2]{matringe2022kirillov}). Recall that
$\mathcal{K}(\overline{\psi}_E)$ denotes the compactly induced
representation ${\rm ind}_{N_n(E)}^{P_n(E)}\overline{\Theta}_E$. Note
that $\mathcal{K}(\overline{\psi}_E)$ is contained in
$\mathbb{K}(J_l(\pi_E),\overline{\psi}_E)$.  Let $I_n$ be the
following natural map:
$$ I_n:\widehat{H}^0(\mathcal{K}(\overline{\psi}_E)) \longrightarrow
\widehat{H}^0(\mathbb{K}(J_l(\pi_E),\overline{\psi}_E)). $$ Let
$\Phi_n:\mathbb{K}(J_l(\pi_E), \overline{\psi}_E)^\Gamma \rightarrow
{\rm Ind}_{N_{n}(F)}^{P_n(F)}\overline{\Theta}_F^l$ be the restriction
to $P_n(F)$ map. Note that the map $\Phi_n$ factorizes through
$$ \Phi_n:\widehat{H}^0(\mathbb{K}(J_l(\pi_E), \overline{\psi}_E))\longrightarrow
{\rm Ind}_{N_{n}(F)}^{P_n(F)}\overline{\Theta}_F^l. $$ The composition
$\Phi_n\circ I_n$ is induced by the restriction to $P_n(F)$ map from
$\mathcal{K}(\overline{\psi}_E)^\Gamma$ to
$\mathcal{K}(\overline{\psi}_F^l)$, and hence, $\Phi_n\circ I_n$ is an
isomorphism onto the space $\mathcal{K}(\overline{\psi}_F^l)$ by
Proposition \ref{tr} (see Subsection \ref{section_kirrilov_rep}).
This implies that the image of $\Phi_n$ contains
$\mathcal{K}(\overline{\psi}_F^l)$. Let $\mathcal{M}(\pi_F, \psi_F)$
be the space
$\Phi_n^{-1}(\mathbb{K}(J_l(\pi_F)^{(l)}, \overline{\psi}_F^l))$.  The
space $\mathcal{M}(\pi_F, \psi_F)$ is a non-zero $P_n(F)$
sub-representation of
$\widehat{H}^0(\mathbb{K}(J_l(\pi_E),\overline{\psi}_E))$, and the map
$$ \Phi_n: \mathcal{M}(\pi_F, \psi_F)\longrightarrow
\mathbb{K}(J_l(\pi_F)^{(l)}, \overline{\psi}_F^l) $$ is
non-zero. Then, using induction on $n$, we will show that the space
$\mathcal{M}(\pi_F, \psi_F)$ is stable under the action of $G_n(F)$
and the map $\Phi_n$ is $G_n(F)$-equivariant.
\subsubsection{}
Let $\overline{J_l(\pi_E)}(w_n)$ be the induced action of
$J_l(\pi_E)(w_n)$ on the space
$\widehat{H}^0(\mathbb{K}(J_l(\pi_E), \overline{\psi}_E))$. Let $V$ be
an element in $\mathcal{M}(\pi_F, \psi_F)$. Then there exists
$W\in\mathbb{W}(J_l(\pi_E),\overline{\psi}_E)^\Gamma$ such that $W$ is
mapped to $V$ under the map
\begin{equation}\label{75}
\mathbb{W}(J_l(\pi_E),\overline{\psi}_E)^\Gamma
\longrightarrow\mathbb{K}(J_l(\pi_E),\overline{\psi}_E)^\Gamma
\longrightarrow \widehat{H}^0\big(\mathbb{K}(J_l(\pi_E), \overline{\psi}_E)\big).
\end{equation}
Let $\overline{\sigma}_F$ be an arbitrary $l$-modular generic
representation of $G_{n-1}(F)$, and let $\sigma_F$ be its $l$-adic
lift. In this case, the generic mod-$l$ representation $J_l(\sigma_F)$
is equal to $\overline{\sigma}_F$. Let $\sigma_E$ be an $l$-adic
generic representation of $G_{n-1}(E)$ obtained as a base change of
$\sigma_F$. Note that the map
$$ \tilde{\Phi}_{n-1}: \widehat{H}^0(\mathbb{W}(J_l(\sigma_E),
\overline{\psi}_E^{-1})) \longrightarrow {\rm
  Ind}_{N_{n-1}(F)}^{G_{n-1}(F)} \overline{\Theta}_F^{-l}$$ is
non-zero. Here, $\tilde{\Phi}_{n-1}$ is the restriction to
$G_{n-1}(F)$ map on the space
${\rm Ind}_{N_{n-1}(E)}^{G_{n-1}(E)}\overline{\Theta}_E$. Assuming the
induction hypothesis for $n-1$ and using the fact that the
representation
$\widehat{H}^0(\mathbb{W}(J_l(\sigma_E), \overline{\psi}_E^{-1}))$ has
a unique generic subquotient (Proposition \ref{tate_is_generic}), the
image of $\tilde{\Phi}_{n-1}$ contains
$\mathbb{W}(\overline{\sigma}_F^{(l)}, \overline{\psi}_F^{-l})$.
Thus, for any
$W'\in \mathbb{W}(\overline{\sigma}_F^{(l)}, \overline{\psi}_F^{-l})$,
there exists an element
$\mathcal{S} \in
\mathbb{W}(J_l(\sigma_E),\overline{\psi}_E^{-1})^\Gamma$ such that
$\tilde{\Phi}_{n-1}(\mathcal{S}) = W'$ and
$$ \tilde{\Phi}_{n-1}\big(J_l(\sigma_E)(w_{n-1})S\big) =
\overline{\sigma}_F^{(l)}(w_{n-1})W'. $$ 
Now the functional equation in
(\ref{w}) gives the following relation :
\begin{equation}\label{func_equ_1}
\sum_{r\in\mathbb{Z}}c^E_{r}\big(\widetilde{W},\widetilde
{\mathcal{S}}\big)q_F^{-\frac{r}{2}f}X^{-fr} = 
\omega_{J_l(\sigma_E)}(-1)^{n-2}
\gamma\big(X, J_l(\pi_E), J_l(\sigma_E), \psi_E\big)
\sum_{r \in\mathbb{Z}}c^E_r(W,\mathcal{S}) q_F^{\frac{r}{2}f}X^{fr},
\end{equation}
where $f$ denotes the residue degree of the extension $E/F$. 
Note that $\omega_{\sigma_E}(-1)=\omega_{\sigma_F}(-1)$ as
$l$ is an odd prime. Applying Proposition \ref{1}, we get
\begin{equation}\label{integral_lift}
\begin{split}
&\int_{(X_E^r)^\Gamma}W
\begin{pmatrix}
	g & 0\\
	0 & 1
\end{pmatrix}
\mathcal{S}(g)dg  = \int_{X_F^{\frac{r}{e}}}W
\begin{pmatrix}
	g & 0\\
	0 & 1
\end{pmatrix}
\mathcal{S}(g)dg,\\
\\
&\int_{(X_E^r)^\Gamma}\widetilde{W}
\begin{pmatrix}
	g & 0\\
	0 & 1
\end{pmatrix}
\widetilde{\mathcal{S}}(g)dg =
\int_{X_F^{\frac{r}{e}}}\widetilde{W}
\begin{pmatrix}
	g & 0\\
	0 & 1
\end{pmatrix}
\widetilde{\mathcal{S}}(g)dg,    
\end{split}
\end{equation}
for each $r\in \mathbb{Z}$. Using the above equalities and Remark
\ref{rmk_int2}, the functional equation (\ref{func_equ_1}) becomes
$$ \sum_{r \in \mathbb{Z}}
c^F_r(\widetilde{W},\widetilde{\mathcal{S}})
q_F^{-\frac{r}{2}}X^{-efr}
= \omega_{J_l(\sigma_F)}(-1)^{n-2}
\gamma(X, J_l(\pi_E), J_l(\sigma_E), \overline{\psi}_E)
\sum_{r \in\mathbb{Z}}
c^F_r(W,\mathcal{S}) q_F^{\frac{r}{2}}X^{efr}. $$
Using the modification as in (\ref{int_2}), the above 
equality becomes
\begin{equation}\label{FE_final_1}
\begin{split}
\sum_{r\in\mathbb{Z}} c^F_{-r}\big(\overline{J_l(\pi_E)}
(w_n)W, &\overline{\sigma}_F^{(l)}(w_{n-1})W'\big)
q_F^{-\frac{r}{2}}X^{-lr} =\\
&\omega_{J_l(\sigma_F)}(-1)^{n-2}
\gamma(X,J_l(\pi_E), J_l(\sigma_E),\overline{\psi}_E)
\sum_{r\in \mathbb{Z}} c^F_r(W,W') q_F^{\frac{r}{2}}X^{lr}.
\end{split}  
\end{equation}
		
\subsubsection{}
For any $V\in \mathcal{M}(\pi_F, \psi_F)$, we show that
\begin{equation}\label{TS}
\Phi_n(\overline{J_l(\pi_E)}(w_n)V)=
J_l(\pi_F)^{(l)}(w_n)\Phi_n(V).
\end{equation}
Let $U$ be an element of
$\mathbb{W}(J_l(\pi_F)^{(l)}, \overline{\psi}_F^l)$ such that
${\rm res}_{P_n(F)}(U)$ is equal to $\Phi_n(V)$. By Lemma \ref{van},
the assertion (\ref{TS}) is equivalent to the following equality :
\begin{equation}\label{Claim}
\begin{split}
\sum_{r\in\mathbb{Z}}c^F_{-r}\big(\overline{J_l(\pi_E)}(w_n)W,
&\overline{\sigma}_F^{(l)}(w_{n-1})W'\big)q_F^{-r/2}X^{-r}=\\
&\sum_{r\in\mathbb{Z}}
c^F_{-r}\big(J_l(\pi_F)^{(l)}(w_n)U,\overline{\sigma}_F^{(l)}
(w_{n-1})W'\big)q_F^{-r/2}X^{-r}, 
\end{split}
\end{equation}
for all
$W'\in \mathbb{W}(\overline{\sigma}_F, \overline{\psi}_F^{-l})$ and
for all $l$-modular generic representations $\overline{\sigma}_F$ of
$G_{n-1}(F)$. Now, consider an $l$-modular generic representation
$\overline{\sigma}_F$ of $G_{n-1}(F)$ and take an $l$-adic lift of
$\overline{\sigma}_F$, say $\sigma_F$ (see subsection \ref{SL}). Note
that $J_l(\sigma_F)=\overline{\sigma}_F$. Let $\sigma_E$ be the
$l$-adic generic representation of $G_{n-1}(E)$ obtained as a base
change lift of $\sigma_F$. From the functional equation with its
modifications as in \eqref{int_2}, we have
\begin{align*}
\begin{split}
\sum_{r \in \mathbb{Z}} c^F_{-r}\big(J_l(\pi_F)^{(l)}(w_n)U,
&\overline{\sigma}_F^{(l)}(w_{n-1})W'\big)
q_F^{-\frac{r}{2}}X^{-lr} =\\
&\omega_{J_l(\sigma_F)}(-1)^{n-2}
\gamma\big(X^l,J_l(\pi_F)^{(l)}, \overline{\sigma}_F^{(l)},
\overline{\psi}_F^l\big)\sum_{r \in \mathbb{Z}} c^F_r(U,W')
q_F^{\frac{r}{2}}X^{lr},    
\end{split}
\end{align*}
where we replace the variable $X$ by $X^l$. Note that
${\rm res}_{P_n(F)}(W)$ is equal to ${\rm res}_{P_n(F)}(U)$. Thus,
comparing the above functional equation with (\ref{FE_final_1}), the
relation \eqref{TS} is now equivalent to the following equality:
$$ \gamma(X,J_l(\pi_E), J_l(\sigma_E),\overline{\psi}_E) =
\gamma\big(X^l, J_l(\pi_F)^{(l)},\overline{\sigma}_F^{(l)},
\overline{\psi}_F^l\big). $$
\subsubsection{}\label{main_proof_fourth__part}
Recall that 
$$ \gamma(X,\pi_E,\sigma_E,\psi_E) =
\epsilon(X,\pi_E,\sigma_E,\psi_E)
\frac{L(q_E^{-1}X^{-1},\widetilde{\pi_E},
\widetilde{\sigma_E})}{L(X,\pi_E,\sigma_E)}. $$
Now using the identity in \cite[Proposition 6.9]{MR1007299}, we have
$$ L(X,\pi_E,\sigma_E)=
\prod_{\eta}L(X,\pi_F,\sigma_F\otimes\eta) $$ and
$$ \epsilon(X,\pi_E,\sigma_E,\psi_E) =
\mathcal{C}_{E/F}(\psi_F)^{n(n-1)}
\prod_{\eta}\epsilon(X,\pi_F,\sigma_F\otimes\eta,\psi_F), $$ where
$\eta$ runs over all the characters of the group
$F^\times/{\rm Nr}_{E/F}(E^\times)$, which is isomorphic to
${\rm Gal}(E/F)$ via local class field theory. Here,
$\mathcal{C}_{E/F}(\psi_F)$ is the Langlands constant, defined as in
the proof of Lemma \ref{i}, and $\mathcal{C}_{E/F}(\psi_F)^2 = 1$.
Then the above relations implies that
\begin{equation}\label{base_change_gamma_identity_general}
\gamma(X,\pi_E,\sigma_E,\psi_E) =
\prod_{\eta}\gamma(X,\pi_F,\sigma_F\otimes\eta,\psi_F).
\end{equation} 
Now, using the identity (\ref{gamma_factor_generic_part}), 
we have
$$ r_l(\gamma(X,\pi_E,\sigma_E,\psi_E)) = 
\gamma(X,J_l(\pi_E), J_l(\sigma_E),\overline{\psi}_E) $$
and 
$$ r_l(\gamma(X,\pi_F, \sigma_F\otimes \eta,\psi_E))
= \gamma(X,J_l(\pi_F), \overline{\sigma}_F,\overline{\psi}_F), $$ for
each character $\eta$. Taking mod-$l$ reduction to the identity
(\ref{base_change_gamma_identity_general}) and using these relations,
we get
\begin{equation}\label{most_important_identity}
\gamma(X,J_l(\pi_E), J_l(\sigma_E),\overline{\psi}_E) =
\gamma(X,J_l(\pi_F),\overline{\sigma}_F,\overline{\psi}_F)^l.
\end{equation}
Finally, it follows from Lemma \ref{ar} that
$$ \gamma(X,J_l(\pi_E), J_l(\sigma_E),\overline{\psi}_E) =
\gamma(X^l,J_l(\pi_F)^{(l)},
\overline{\sigma}_F^{(l)},\overline{\psi}_F^l).$$
The identity (\ref{TS}) shows that space
$\mathcal{M}(\pi_F, \psi_F)$ is stable under the
action of $G_n(F)$ and the map
$$ \Phi_n: \mathcal{M}(\pi_F, \psi_F)\rightarrow
\mathbb{K}(J_l(\pi_F)^{(l)}, \overline{\psi}^l_F) $$ is
surjective. Using Proposition \ref{tate_is_generic}, the $G_n(F)$
representation
$\widehat{H}^0(\mathbb{W}(J_l(\pi_E), \overline{\psi}_E))$ has a
unique generic sub-quotient, which is necessarily equal to
$J(\pi_F)^{(l)}$. This completes the proof.
\end{proof}
Now we deduce some corollaries of Theorem \ref{noncuspidal_n_thm}. We
keep the same assumptions that $E/F$ is a finite Galois extension
$p$-adic fields with $[E:F]=l$, where $l$ and $p$ are distinct primes,
and $l$ does not divide $\lvert G_{n-1}(k_F)\rvert$.
\begin{corollary}\label{noncuspidal_n_Whittaker_lattice_thm}
  Let $\pi_E$ be an integral generic $\mathcal{K}$-representation of
  $G_n(E)$ which is absolutely irreducible. Assume that
  $\pi_E^\gamma \simeq \pi_E$, for all $\gamma\in \Gamma$. Let
  $\mathbb{W}_\Lambda(\pi_E,\psi_E)$ be the space of all
  $\Lambda$-valued functions in the Whittaker model of $\pi_E$.  Let
  $\pi_F$ be the integral generic
  $\overline{\mathbb{Q}}_l$-representation of $G_n(F)$ such that
  $\pi_E\otimes_{\mathcal{K}} \overline{\mathbb{Q}}_l$ is the base
  change lift of $\pi_F$. Then the Frobenius twist of $J_l(\pi_F)$
  occurs as a unique generic subquotient of the zeroth Tate cohomology
  group $\widehat{H}^0(\mathbb{W}_\Lambda(\pi_E,\psi_E))$.
\end{corollary}
\begin{proof}
  The outline of the proof is same as Theorem
  \ref{noncuspidal_n_thm}. For the sake of completeness, we discuss
  some crucial steps. As one can observe, the previous and the present
  theorems are similar in spirit to local converse theorem for
  $(n,n-1)$, we precisely use Theorem \ref{noncuspidal_n_thm} at the
  $(n-1)$ step.
		
First, note that the $\Lambda$-lattice
$\mathbb{W}_\Lambda(\pi_E,\psi_E)$ is stable under the action of
$G_n(E)\rtimes \Gamma$ (see \cite[Theorem 2]{MR2058628} and
Lemma \ref{inv_whit_model}). Consider the integral Kirillov model
$\mathbb{K}_\Lambda(\pi_E,\psi_E)$. The restriction map
$W\mapsto {\rm res}_{P_n(E)}(W)$ is then a bijection from
$\mathbb{W}_\Lambda(\pi_E,\psi_E)$ onto
$\mathbb{K}_\Lambda(\pi_E,\psi_E)$. Let $\Phi_n$ be the following
$P_n(F)$-equivariant map, defined as the composition of restriction to
$P_n(F)$ map and (pointwise) mod-$l$ reduction map
$$ \Phi_n : \mathbb{K}_\Lambda(\pi_E,\psi_E)^\Gamma
\longrightarrow \mathbb{K}(J_l(\pi_F)^{(l)},\overline{\psi}_F^l). $$
Then $\Phi_n$ is non-zero and it factorizes through the Tate
cohomology space $\widehat{H}^0(\mathbb{K}_\Lambda(\pi_E,\psi_E))$. As
before, we consider the non-zero space
$\Phi_n^{-1}(\mathbb{K}(J_l(\pi_F)^{(l)}, \overline{\psi}_F^l))$ and
denote it by $\mathcal{M}(\pi_F, \psi_F)$. To prove the above
corollary, it is sufficient prove that
$$ \Phi_n(\overline{\pi_E(w_n)}V) = J_l(\pi_F)^{(l)}(w_n)\Phi_n(V), $$
for all $V\in \mathcal{M}(\pi_F,\psi_F)$. It is enough to prove the
following identity of Laurent series
\begin{equation}\label{Laurent_series_identity}
\sum_{r\in\mathbb{Z}}c^F_{r}\big(\Phi_n(
\overline{\pi_E(w_n)}
V),W'\big)q_F^{r/2}X^{r} = \sum_{r\in\mathbb{Z}}
c^F_{r}\big(J_l(\pi_F)^{(l)}(w_n)\Phi_n(V),W'\big)
q_F^{r/2}X^{r},
\end{equation}
for all
$W'\in \mathbb{W}(\overline{\sigma}_F^{(l)},\overline{\psi}_F^{-l})$
and for all $l$-modular generic representations $\overline{\sigma}_F$
of $G_{n-1}(F)$. Take such mod-$l$ generic representation
$\overline{\sigma}_F$ of $G_{n-1}(F)$. Let $\sigma_F$ be an $l$-adic
lift of $\overline{\sigma}_F$ and let $\sigma_E$ be the base change
lift of $\sigma_F$ to $G_{n-1}(E)$. Theorem \ref{noncuspidal_n_thm}
gives a $G_{n-1}(F)$-stable subspace $\mathcal{N}(\sigma_F,\psi_F)$ of
the Tate cohomology group
$\widehat{H}^0(\mathbb{W}(J_l(\sigma_E),\overline{\psi}_E))$ with the
following $G_{n-1}(F)$-equivariant surjection
$$ \Phi_{n-1} : \mathcal{N}(\sigma_F,\psi_F)
\longrightarrow
\mathbb{W}(\overline{\sigma}_F^{(l)},\overline{\psi}_F^{-l}). $$ Now,
lifting the function $W'$ via $\Phi_{n-1}$ and the function $V$ to the
respective $\Gamma$-invariant Kirillov models and using the identities
(\ref{integral_lift}), the above relation
(\ref{Laurent_series_identity}) is then equivalent to the following
identity of gamma factors :
$$ r_l(\gamma(X,\pi_E,\sigma_E,\psi_E)) =
\gamma(X^l, J_l(\pi_F)^{(l)}, \overline{\sigma}_F^{(l)},
\overline{\psi}_F^l). $$ This follows from the arguments of the
subsection (\ref{main_proof_fourth__part}) in the proof of 
Theorem \ref{noncuspidal_n_thm}.
\end{proof}
\begin{corollary}\label{n_thm}
Let $n\geq 3$, and 
let $\pi_F$ be an integral $l$-adic
cuspidal representation of $G_n(F)$ and let $\pi_E$ be 
the base change lift of
$\pi_F$. We further assume that $l$ does not divide $n$. 
Then we have 
$$\widehat{H}^0(r_l(\pi_E))\simeq r_l(\pi_F)^{(l)}.$$
\end{corollary}
\begin{proof}
Since $l$ does not divide $n$, the representation $\pi_E$ is
cuspidal. As the Kirillov model
$\mathbb{K}(r_l(\pi_E), \overline{\psi}_E)$ is equal to
$\mathcal{K}(\overline{\psi}_E)$, we get that
$\widehat{H}^0(\mathbb{K}(r_l(\pi_E), \overline{\psi}_E))$ is equal
to $\mathcal{K}(\overline{\psi}_F^l)$. Thus, the action of $G_n(F)$
on $\widehat{H}^0(\mathbb{K}(r_l(\pi_E), \overline{\psi}_E))$ is
irreducible, and the corollary follows from Theorem
\ref{noncuspidal_n_thm}.
\end{proof}
	
\section{Base change for
\texorpdfstring{$Z(\Delta)$}{}}\label{Tate_Zelevinsky}
In this section, we study the Tate cohomology of the base change of
the Zelevinsky subrepresentations of the form $Z(\Delta)$. In
\cite{MR584084}, Zelevinsky uses the notation $\langle\Delta\rangle$
for $Z(\Delta)$. In this section, we continue with the assumptions in
Corollary \ref{n_thm}, i.e., $l\neq p$ and $l$ does not divide
$|G_{n-1}(\mathbb{F}_q)|$ and the integer $n$. Recall that $q$ is the
cardinality of the residue field of $F$. We will crucially use the
fact that $Z(\Delta)$ remains irreducible under the restriction to
$P_n$ and it is characterised by this property.
\subsection{}\label{intro}
Keeping the notations as in subsection (\ref{segment}), let
$\Delta=\big\{\sigma,\sigma\nu_K,\dots,\sigma\nu_K^{r-1}\big\}$ be a
segment, where $K$ is a $p$-adic field and $\sigma$ is a cuspidal
$l$-adic representation of $G_m(K)$. We denote by $\ell(\Delta)$ the
length of $\Delta$, i.e., the integer $r$. The parabolic induction
$$ \sigma\times\sigma\nu_K\times\cdots\sigma\nu_K^{r-1} $$
admits a unique irreducible subrepresentation, denoted by
$Z(\Delta)$. Moreover, $Z(\Delta)$ can be characterised as those
irreducible representation of $G_{rm}(K)$ that remain irreducible
after restricting to $P_{rm}(K)$, and the restriction is isomorphic to
$(\Phi^+)^{m-1}\circ\Psi^+(Z(\Delta^-))$, where
$\Delta^-=\Delta\setminus\big\{\sigma\nu_K^{r-1}\big\}$. We refer to
\cite[Section 3]{MR579172}, \cite[Section 1.2, Chapter III]{MR1395151}
for the definitions of the functors $\Phi^{\pm}$ and $\Psi^{\pm}$, and
for the definition of $Z(\Delta)$ and its restriction to $P_n(K)$, we
refer to \cite[Section 3]{MR584084}.
\subsection{}\label{hyp_1}
Let $F$ be a finite extension of $\mathbb{Q}_p$, and let $E$ be a
finite Galois extension of $F$ of prime degree $l$ with $l\ne p$.  Let
$\Gamma$ denote the Galois group ${\rm Gal}(E/F)$ with generator, say
$\gamma$. Let $\sigma_F$ and $\sigma_E$ be the integral cuspidal
$l$-adic representations of $G_m(F)$ and $G_m(E)$ respectively, such
that $\sigma_E$ is a base change lift of $\sigma_F$. Consider the
segments
$$ \Delta_F = \big\{\sigma_F,\sigma_F\nu_F,
\dots,\sigma_F\nu_F^{k-1}\big\} $$
$$ \Delta_E=\big\{\sigma_E,\sigma_E\nu_E,\dots,\sigma_E
\nu_E^{k-1}\big\}. $$ 
Then we have the irreducible $l$-adic
representations $Z(\Delta_F)$ and $Z(\Delta_E)$ of $G_n(F)$ and
$G_n(E)$ respectively, where $n=km$. If we let $\sigma_F'$
(resp. $\sigma_E'$) to be the representation 
$\sigma_F\nu_F^{k-1}$
(resp. $\sigma_E\nu_E^{k-1}$), then we have
$$ \Pi_F(Z(\Delta_F))=\Pi_F(\sigma_F')\oplus
\Pi_F(\sigma_F'\nu_F^{-1})\oplus\cdots\oplus\Pi_F(\sigma_F) $$ 
and
$$ \Pi_E(Z(\Delta_E))=\Pi_E(\sigma_E')\oplus
\Pi_E(\sigma_E'\nu_E^{-1})\oplus\cdots\oplus\Pi_F(\sigma_E), $$ 
where $\Pi_F$ and $\Pi_E$ are the local Langlands 
correspondences
defined as in subsection (\ref{LLC}). This shows that
$$ {\rm Res}_{\mathcal{W}_E}\big(\Pi_F(Z(\Delta_F))\big)
\simeq\Pi_E(Z(\Delta_E)). $$
Thus the representation $Z(\Delta_E)$ is the base change of
$Z(\Delta_F)$.
\subsection{}\label{integral_structure}
Let $\mathcal{L}_0$ be a $G_m(E)$-invariant lattice in $\sigma_E$, and
let $S_\gamma:\sigma_E\rightarrow\sigma_E^\gamma$ be an isomorphism
with $S_\gamma^l={\rm id}$ and
$S_\gamma(\mathcal{L}_0)=\mathcal{L}_0$. Recall that the
representation
$\pi_E=\sigma_E\times\sigma_E\nu_E\times\cdots\times
\sigma_E\nu_E^{k-1}$
admits a $G_n(E)$-invariant lattice, say $\mathcal{L}'$, which is
induced via $\mathcal{L}_0$. 
Then $\mathcal{L}=\mathcal{L}'\cap Z(\Delta_E)$ is a
$G_n(E)$-invariant lattice in $Z(\Delta_E)$. Now, the map $S_\gamma$
induces an isomorphism
$T_\gamma:Z(\Delta_E)\rightarrow Z(\Delta_E)^\gamma$ such that
$T_\gamma^l={\rm id}$ and $T_\gamma$ stabilizes
$\mathcal{L}$. Moreover, choosing a $G_m(E)$-invariant,
$S_\gamma$-stable lattice in $\sigma_E$ is equivalent to choosing a
$G_n(E)$-invariant, $T_\gamma$-stable lattice in $Z(\Delta_E)$.
\begin{remark}\normalfont\label{Tate_cusp_vanish}
Since $\sigma_E$ is cuspidal, the restriction $\sigma_E|_{P_m(E)}$
is isomorphic to the compact induction $\mathcal{K}(\psi_E)$ as
$\overline{\mathbb{Q}}_l$ representations. So, the restriction of
$\mathcal{L}_0$ to the subgroup $P_m(E)$ is isomorphic to the space
of $\overline{\mathbb{Z}}_l$ valued functions in
$\mathcal{K}(\psi_E)$. This implies that
$\widehat{H}^1(\mathcal{L}_0) = 0$. For details, see \cite[Theorem
6]{MR3551160}.
\end{remark}
From this, we now deduce the following result. 
\begin{proposition}\label{Tate_vanish}
Let $\mathcal{L}$ be a lattice in $Z(\Delta_E)$ that is stable under
the action of both $G_n(E)$ and $T_\gamma$. Then we have
$\widehat{H}^1(\mathcal{L})=0$.
\end{proposition}
\begin{proof}
  We prove the above claim using induction on $\ell(\Delta_E)$. If the
  length of $\Delta_E$ is $1$, then the proposition clearly follows
  from Remark \ref{Tate_cusp_vanish}. Recall that
$$ 
Z(\Delta_E)|_{P_n(E)}\simeq(\Phi^+)^{m-1}
\circ\Psi^+(Z(\Delta_E^-)).
$$
Here, $Z(\Delta_E)^-$ is identified with the $m$-th derivative 
of $Z(\Delta_E)$. The $m$-th derivative is the composition of
the following maps:
$$Z(\Delta_E)\xrightarrow{r_{N_{n-m, m}(E)}} Z(\Delta_E^-)\otimes \sigma_E
\xrightarrow{\id\otimes r_{N_m(E), \Theta_E}} Z(\Delta_E^-),$$ where
$r_{N_{n-m, m}(E)}$ and $\id \otimes r_{N_m(E), \Theta_E}$ are the
natural quotient maps of corresponding Jacquet module and twisted
Jacquet module. Note that the maps $r_{N_{n-m, m}(E)}$ and
$r_{N_m(E), \Theta_E}$ preserve integral structures (see
\cite[Proposition 1.4(i)]{MR2120114} for $r_{N_{n-m, m}(E)}$, and
\cite[Theorem III.2]{MR2058628} for $r_{N_m(E), \Theta_E}$). Thus, we
get that $\mathcal{L}^{-}$ defined as
$$\mathcal{L}^{-}=\Psi^{-}\circ (\Phi^-)^{m-1}(\mathcal{L})$$
is a lattice in $Z(\Delta_E^-)$. 
We have the following isomorphism of $P_n(E)\rtimes\Gamma$-modules
$$ 
\mathcal{L}|_{P_n(E)}\simeq(\Phi^+)^{m-1}
\circ\Psi^+(\mathcal{L}^-).
$$
Applying
Proposition \ref{tr} repeatedly $(m-1)$-times, we get that
\begin{equation}\label{induction_Z(D)_bla}
\widehat{H}^1(\mathcal{L})\simeq(\Phi^+)^{m-1}\circ\Psi^+
\big(\widehat{H}^1(\mathcal{L}^-)\big).
\end{equation}
When $k$ is $2$, we have $\Delta_E=\big\{\sigma_E,\sigma_E\nu_E\big\}$
and in this case, the representation $Z(\Delta_E^-)$ is equal to
$\sigma_E$ and $\mathcal{L}^-$ is a $G_m(E)\rtimes \Gamma$-stable
lattice in $\sigma_E$. Then, it follows from Remark
\ref{Tate_cusp_vanish} and the isomorphism (\ref{induction_Z(D)_bla})
that $\widehat{H}^1(\mathcal{L})=0$. Suppose the result is true for
all $Z(\Delta)$'s where the length of $\Delta$ is strictly less than
$k$. Recall that the length of $\Delta_E^-$ is $k-1$. By induction
hypothesis, we have $\widehat{H}^1(\mathcal{L}^-)=0$. Then, using
(\ref{induction_Z(D)_bla}), we get that
$\widehat{H}^1(\mathcal{L})=0$.
\end{proof}
\subsection{}
We now recall the mod-$l$ reduction of the representation
$Z(\Delta_F)$. Let us introduce the following notations:
$$ r_l(\Delta_F)=\big\{r_l(\sigma_F),r_l(\sigma_F)
\overline{\nu}_F,\dots,r_l(\sigma_F)\overline{\nu}_F^{k-1}\big\} $$
and
$$ r_l(\Delta_F)^{(l)} =\big\{r_l(\sigma_F)^{(l)},(r_l(\sigma_F)
\overline{\nu}_F)^{(l)},\dots,(r_l(\sigma_F)
\overline{\nu}_F^{k-1})^{(l)}\big\}, $$ 
where $r_l(\sigma_F)$ is the mod-$l$ reduction of $\sigma_F$ 
and $r_l(\sigma_F)^{(l)}$ is the
Frobenius twist of $r_l(\sigma_F)$. Then the mod-$l$ reduction of
$Z(\Delta_F)$ (\cite[Theorem 9.39]{MR3178433}) is given by
$$ r_l(Z(\Delta_F))=Z(r_l(\Delta_F)). $$
This shows, in particular, that $r_l(Z(\Delta_F))$ is
irreducible. Moreover, the Frobenius twist of $r_l(Z(\Delta_F))$
equals $Z(r_l(\Delta_F)^{(l)})$.  We conclude this section with the
following theorem.
\begin{theorem}\label{Tate_Z(D)}
Let $E/F$ be a finite Galois extension with $[E:F]=l$, where $l$ and
$p$ are distinct primes such that $l$ does not divide $n$ and
$|G_{n-1}(\mathbb{F}_q)|$. Let $\sigma_F$ be an integral cuspidal
$l$-adic representation of $G_m(F)$, and let $\sigma_E$ be an
integral $l$-adic representation of $G_m(E)$ obtained as a base
change of $\sigma_F$ (Note that $\sigma_E$ is also cuspidal). Let
$\Delta_F=\big\{\sigma_F,\sigma_F\nu_F,\dots,
\sigma_F\nu_F^{k-1}\big\}$
and
$\Delta_E=\big\{\sigma_E,\sigma_E\nu_E,\dots,
\sigma_E\nu_E^{k-1}\big\}$
be two segments (Here $n=km$). Then we have
$$ \widehat{H}^0(r_l(Z(\Delta_E)))\simeq r_l(Z(\Delta_F))^{(l)}. $$
\end{theorem}
\begin{proof}
We use induction on $\ell(\Delta_E)$. For $k=1$, we have
$Z(\Delta_E)=\sigma_E$ and $Z(\Delta_F)=\sigma_F$, and the theorem
follows from Corollary \ref{n_thm}. Suppose the result is true for
all segments $\Delta_F'$ and $\Delta_E'$ with
$\ell(\Delta_F')=\ell(\Delta_E') < k$. Let $\tau_F$ and $\tau_E$ be
the mod-$l$ Zelevinsky representations $Z(r_l(\Delta_F))$ and
$Z(r_l(\Delta_E))$, respectively. We denote by $\tau_F^-$ and
$\tau_E^-$ the mod-$l$ representations $Z(r_l(\Delta_F^-))$ and
$Z(r_l(\Delta_E^-))$ respectively. Since the restriction
$\tau_E|_{P_n(E)}$ is isomorphic to
$(\Phi^+)^{m-1}\circ\Psi^+(\tau_E^-)$, it follows from Proposition \ref{tr} that
\begin{equation}\label{isom_Zel}
\widehat{H}^0(\tau_E|_{P_n(E)}) \simeq
(\Phi^+)^{m-1}\circ\Psi^+(\widehat{H}^0(\tau_E^-)).
\end{equation}
By induction hypothesis, we have
$$ \widehat{H}^0(\tau_E^-) \simeq  (\tau_F^-)^{(l)}. $$ 
Thus it follows from (\ref{isom_Zel}) and \cite[Chapter 3,
1.5]{MR1395151} that $\widehat{H}^0(\tau_E)$ is an irreducible
representation of $P_n(F)$ and hence irreducible as a representation
of $G_n(F)$. Let $\lambda = (m,m,\dots,m)$ be the partition of $n$ and
let $P_\lambda=M_\lambda N_\lambda$ be the parabolic subgroup of
$G_n$. The isomorphism \label{isom_Z_2} implies that
$\widehat{H}^0(\tau_E)_{N_\lambda(F)}$ is non-zero. Then, using Lemma
\ref{Jaquet_Tate} and Corollary \ref{n_thm}, we get the following
isomorphism of $M_\lambda(F)$-representations
\begin{equation}\label{Jacquet_Tate_Z(D)}
(\widehat{H}^0(\tau_E))_{N_\lambda(F)}
\simeq ((\tau_F^-)^{(l)})_{N_\lambda(F)}.
\end{equation}
The irreducibility of $\widehat{H}^0(\tau_E)$ and the 
isomorphism (\ref{Jacquet_Tate_Z(D)}) implies that
$$ \widehat{H}^0(\tau_E) \simeq \tau_F^{(l)} $$
as a representation of $G_n(F)$ (see \cite[Proposition
V.9.1]{Vigneras_Induced}).
\end{proof}
	
\section{Irreducibility of Tate Cohomology of generic
representations}\label{Tate_generic}
In this section, we discuss the Tate cohomology groups of
representations of the form $L(\Delta)$, where
$L(\Delta)$ is defined in subsection (\ref{segment}). We
assume that $l$ does not divide the pro-order of
$G_n(F)$. We continue with the notation
that $\sigma_F$ is an $l$-adic cuspidal representation
of $G_n(F)$ and $\sigma_E$ is the base change lift of
$\pi_F$ to $G_n(E)$.
\subsection{}\label{integral_structure_2}
Keep the notations as in subsection (\ref{hyp_1}). Recall that
$L(\Delta_E)$ is the unique generic quotient of the parabolically
induced representation
$\sigma_E\times\sigma_E\nu_E\times\cdots\times\sigma_E\nu_E^{k-1}$.
Now fix a $G_m(E)$-invariant lattice $\mathcal{L}_0$ in
$\sigma_E$. Then we have the $G_n(E)$-invariant lattice
$\mathcal{L}_0\times\cdots\times\mathcal{L}_0$ in
$\sigma_E\times\cdots\times\sigma_E\nu_E^{k-1}$, and the image of
$\mathcal{L}_0\times\cdots\times\mathcal{L}_0$ under the surjection
$$ \sigma_E\times\sigma_E\nu_E\times\cdots\times
\sigma_E\nu_E^{k-1}\longrightarrow
L(\Delta_E), $$ 
say $\mathcal{L}$, is again a $G_n(E)$-invariant
lattice in $L(\Delta_E)$. As
in subsection (\ref{integral_structure}), an isomorphism between
$\sigma_E$ and $\sigma_E^\gamma$ induces an isomorphism
$T_\gamma:L(\Delta_E)\rightarrow L(\Delta_E)^\gamma$ with
$T_\gamma^l={\rm id}$ and $T_\gamma(\mathcal{L})=\mathcal{L}$.  Here,
the group $\Gamma$ acts on the lattice $\mathcal{L}$ by $T_\gamma$.
\begin{proposition}\label{Tate_1_L(D)_vanish}
Let $\mathcal{L}$ be a lattice in $L(\Delta_E)$ that is stable under
the action of $G_n(E)$ and $T_\gamma$. Then
$\widehat{H}^1(\mathcal{L})=0$.
\end{proposition}
\begin{proof}
  We proceed by induction on $\ell(\Delta_E)$, which equals $k$. When
  $\ell(\Delta_E)=1$, then $L(\Delta_E)=\sigma_E$. In this case, the
  proposition follows from \cite[Theorem 6]{MR3551160}. Suppose the
  result is true for all representations $L(\Delta)$, where
  $\ell(\Delta)$ is strictly less than $k$. Let $\tau$ be the
  restriction ${\rm res}_{P_n(E)}(L(\Delta_E))$. Consider the
  filtration of $P_n(E)$-representations:
$$
(0)\subseteq\tau_n\subseteq\cdots
\subseteq\tau_2\subseteq\tau_1 = \tau,
$$
where $\tau_i/\tau_{i+1}=(\Phi^+)^{i-1}\circ\Psi^+(\tau^{(i)})$ and
$\tau^{(i)}=\Psi^-\circ (\Phi^-)^{i-1}(\tau)$. The
map $T_\gamma$ induces an isomorphism between $\tau^{(i)}$ and
$(\tau^{(i)})^\gamma$, and also between the representations $\tau_i$
and $\tau_i^\gamma$. Hence, there is an action of $\Gamma$ on both
$\tau^{(i)}$ and $\tau_i$. From
\cite[Proposition 9.6]{MR584084}, we get that
\begin{center}
$\tau^{(j)}=0$, if $j$ is not divisible by $m$, and
\end{center}
\begin{center}
$\tau^{(rm)} = L\big(\big\{\sigma_E\nu_E^r,\dots,\sigma_E\nu_E^{k-1}
\big\}\big)$, for $r=0,1,\dots,k-1$.
\end{center}
For each $r\in\{1,2\dots,k-1\}$, let $\Delta_E'$ and $\Delta_E''$ be
the segments $\{\sigma_E\nu_E^r,\dots,\sigma_E\nu_E^{k-1}\}$ and
$\{\sigma_E,\dots,\sigma_E\nu_E^{r-1}\}$ respectively. The $rm$-th
derivative of $L(\Delta_E)$ is the composition of the following maps:
$$ \tau \xrightarrow{r_{N_{n-rm, rm}(E)}} L(\Delta_E')\otimes L(\Delta_E'')
\xrightarrow{\id\otimes r_{N_{rm}(E), \Theta_E}} \tau^{(rm)},$$ where
$r_{N_{n-rm, rm}(E)}$ and $\id \otimes r_{N_{rm}(E), \Theta_E}$ are
the quotient maps of the corresponding Jacquet module and twisted
Jacquet module. Note that the maps $r_{N_{n-rm, rm}(E)}$ and
$r_{N_{rm}(E), \Theta_E}$ preserve integral structures (see
\cite[Proposition 1.4(i)]{MR2120114} and \cite[Theorem
III.2]{MR2058628}). Thus, we get that $\mathcal{L}^{(rm)}$, defined as
$$ \mathcal{L}^{(rm)} = \Psi^-\circ (\Phi^-)^{rm-1}(\mathcal{L}), $$
is a $G_{n-rm}(E)\rtimes \Gamma$-invariant lattice in
$L(\Delta_E)^{(rm)}$.
Let $\mathcal{L}_i
\subset\tau_i$ be the
$P_n(E)\rtimes \Gamma$-invariant $\overline{\mathbb{Z}}_l$-lattice 
$$\mathcal{L}_i=(\Phi^+)^{i-1}\circ(\Phi^-)^{i-1}(\mathcal{L}),$$
for all $1\leq i\leq n$. For $1\leq r\leq k-1$, we have the short
exact sequence of $P_n(E)\rtimes\Gamma$-modules
\begin{equation}\label{derivates}
0\longrightarrow\mathcal{L}_{(r+1)m}\longrightarrow\mathcal{L}_{rm}
\longrightarrow(\Phi^+)^{rm-1}\circ\Psi^+(\mathcal{L}^{(rm)})
\longrightarrow 0
\end{equation}
By induction
hypothesis, we have $\widehat{H}^1(\mathcal{L}^{(rm)})=0$. Then the
long exact sequence of Tate cohomology corresponding to (\ref{derivates}) gives
$$\cdots\longrightarrow\widehat{H}^1(\mathcal{L}_{(r+1)m})
\longrightarrow\widehat{H}^1(\mathcal{L}_{rm})\longrightarrow 0
\longrightarrow\widehat{H}^0(\mathcal{L}_{(r+1)m})
\longrightarrow\cdots$$
For $r=k-1$, the representation
$\tau_n$, is equal to
${\rm ind}_{N_n(E)}^{P_n(E)}\Theta_E$. In this
case, $\widehat{H}^1(\mathcal{L}_n)=0$ by Proposition 
\ref{tr}. Then, from the above long exact sequence, we get 
that $\widehat{H}^1(\mathcal{L}_{(k-1)m})=0$. Again using the above 
long exact sequence for $r=k-2$, we get that
$\widehat{H}^1(\mathcal{L}_{(k-2)m})=0$. Thus, an 
inductive process gives
$$
\widehat{H}^1(\mathcal{L}) = 
\widehat{H}^1(\mathcal{L}_m) = 0.
$$
\end{proof}
\subsection{}
Let $\pi_E$ be a generic, integral $l$-adic representation of
$G_n(E)$. Then $\pi_E$ is of the form
$$
\mathcal{L}(\Delta_1)\times\mathcal{L}(\Delta_2)
\times\cdots\times\mathcal{L}(\Delta_t),
$$ 
where for each
$j\in\{1,2,\dots,t\}$, the representation $\mathcal{L}(\Delta_j)$ is
integral. Let $\mathcal{L}_j$ be a lattice in $L(\Delta_j)$, defined
as in subsection (\ref{integral_structure_2}). Let $T_{\gamma,j}$ be
the isomorphism between $L(\Delta_j)$ and $L(\Delta_j)^\gamma$ such
that $T_{\gamma,j}(\mathcal{L}_j)=\mathcal{L}_j$. Now consider the
$\overline{\mathbb{Z}}_l$-module
$\mathcal{L}=\mathcal{L}_1\times\cdots\times\mathcal{L}_t$. Then
$\mathcal{L}$ is a lattice in $\pi_E$ that is stable under the action
of $G_n(E)$. Moreover, we have an isomorphism
$T_\gamma:\pi_E\rightarrow\pi_E^\gamma$, induced by
$\big\{T_{\gamma,j}\big\}_{j=1}^t$, such that
$T_\gamma(\mathcal{L})=\mathcal{L}$. Note that $l$ is 
banal for $G_n(E)$. Since the mod-$l$
reduction of $\pi_E$ is irreducible, any lattice 
in $\pi_E$ is homothetic to $\mathcal{L}$.
\begin{corollary}\label{Tate_1_generic}
Assume that $l$ does not divide
$|G_{n}(\mathbb{F}_q)|$. Let $\pi_E$ be a generic, integral $l$-adic
representation of $G_n(E)$ as above. Let $\mathcal{L}$ be a lattice
in $\pi_E$ that is stable under the action of $G_n(E)$ and
$T_\gamma$. Then $\widehat{H}^1(\mathcal{L})=0$.
\end{corollary}
\begin{proof}
Using Proposition \ref{tr}, we have
$$ \widehat{H}^1(\mathcal{L})=\widehat{H}^1(\mathcal{L}_1)
\times\cdots\times\widehat{H}^1(\mathcal{L}_t). $$
Now applying Proposition \ref{Tate_1_L(D)_vanish}, we get that
$\widehat{H}^1(\mathcal{L}_i)=0$, for each $i$. Hence, the theorem.
\end{proof}
\begin{theorem}\label{Tate_L(D)}
Let $E/F$ be a finite Galois extension with $[E:F]=l$, where $l$ and
$p$ are distinct primes such that $l$ does not divide
$|G_n(\mathbb{F}_q)|$. Let $\sigma_F$ be an integral cuspidal
$l$-adic representation of $G_m(F)$, and let $\sigma_E$ be an
integral cuspidal $l$-adic representation of $G_m(E)$ obtained as a
base change of $\sigma_F$. Let
$\Delta_F=\big\{\sigma_F,\sigma_F\nu_F,\dots,
\sigma_F\nu_F^{k-1}\big\}$
and
$\Delta_E=\big\{\sigma_E,\sigma_E\nu_E,\dots,
\sigma_E\nu_E^{k-1}\big\}$
be two segments (Here $n=km$). Then
$$ \widehat{H}^0(r_l(L(\Delta_E)))\simeq r_l(L(\Delta_F))^{(l)}. $$
\end{theorem}
\begin{proof}
We prove the theorem using induction on $\ell(\Delta_F)$. Since $l$
does not divide $|G_n(\mathbb{F}_q)|$, the mod-$l$ reduction of the
irreducible integral representations $L(\Delta_F)$ and $L(\Delta_E)$
are also irreducible, and we have
$$ r_l(L(\Delta_F)) = L(r_l(\Delta_F)) $$
and 
$$ r_l(L(\Delta_E)) = L(r_l(\Delta_E)) $$
where $r_l(\Delta)$'s are defined as in subsection
(\ref{integral_structure}). Using the long exact sequence in Tate
cohomology for the exact sequence \eqref{derivates} we get a
filtration
$$ \res_{P_n(F)}\widehat{H}^0(r_l(L(\Delta_E)))= \eta_1\supseteq
\eta_2\supseteq \cdots\supseteq \eta_n, $$ such that
$\eta_i/\eta_{i+1}\neq 0$ if and only if $i$ is a multiple of $m$. By
induction hypothesis, we get that $\eta_{ms}/\eta_{m(s+1)}$ is an
irreducible representation of $P_n(F)$. Theorem
\ref{noncuspidal_n_thm} says that the Frobenius twist
$r_l(L(\Delta_F))^{(l)}$ is the unique generic sub-quotient of
$\widehat{H}^0(r_l(L(\Delta_E)))$. Since the lengths of $P_n(F)$
representations $\widehat{H}^0(r_l(L(\Delta_E)))$ and
$r_l(L(\Delta_F))$ are the same, we get that $r_l(L(\Delta_F))^{(l)}$
is isomorphic to $\widehat{H}^0(r_l(L(\Delta_E)))$.
\end{proof}
Let us continue with the hypothesis as in Theorem \ref{Tate_L(D)}. Let
$\pi$ be an integral $l$-adic generic smooth representation of
$G_n(E)$. Since $l$ does not divide $|G_n(\mathbb{F}_q)|$, the
mod-$l$-reduction $r_l(\pi)$ is irreducible, and hence generic. Then
we have
\begin{corollary}\label{Tate_0_generic}
Let $E/F$ be a finite Galois extension with $[E:F]=l$, where $l$ and
$p$ are distinct primes such that $l$ does not divide
$|G_n(\mathbb{F}_q)|$. Let $\pi_F$ be an integral $l$-adic generic
representation of $G_n(E)$, and let $\pi_E$ be a base change lift of
$\pi_F$ (Note that $\pi_E\simeq\pi_E^\gamma$).  Then
$$ \widehat{H}^0(r_l(\pi_E)) \simeq r_l(\pi_F)^{(l)}. $$ 
\end{corollary}
\begin{proof}
This follows from Proposition \ref{tr} and Theorem \ref{Tate_L(D)}. 
\end{proof}

	\bibliographystyle{amsalpha}
	\bibliography{rev_3_imrn.bib}

\providecommand{\bysame}{\leavevmode\hbox to3em{\hrulefill}\thinspace}
\providecommand{\MR}{\relax\ifhmode\unskip\space\fi MR }
\providecommand{\MRhref}[2]{%
  \href{http://www.ams.org/mathscinet-getitem?mr=#1}{#2}
}
\providecommand{\href}[2]{#2}
\begin{thebibliography}{DHKM24}

\bibitem[AC89]{MR1007299}
James Arthur and Laurent Clozel, \emph{Simple algebras, base change, and the
  advanced theory of the trace formula}, Annals of Mathematics Studies, vol.
  120, Princeton University Press, Princeton, NJ, 1989. \MR{1007299}

\bibitem[BH03]{MR1981032}
Colin~J. Bushnell and Guy Henniart, \emph{Generalized {W}hittaker models and
  the {B}ernstein center}, Amer. J. Math. \textbf{125} (2003), no.~3, 513--547.
  \MR{1981032}

\bibitem[BH06]{MR2234120}
\bysame, \emph{The local {L}anglands conjecture for {$\rm GL(2)$}}, Grundlehren
  der mathematischen Wissenschaften [Fundamental Principles of Mathematical
  Sciences], vol. 335, Springer-Verlag, Berlin, 2006. \MR{2234120}

\bibitem[BZ76]{MR0425030}
I.~N. Bern\v{s}te\u{\i}n and A.~V. Zelevinski\u{\i}, \emph{Representations of
  the group {${\rm GL}(n,F),$} where {$F$} is a local non-{A}rchimedean field},
  Uspehi Mat. Nauk \textbf{31} (1976), no.~3(189), 5--70. \MR{0425030}

\bibitem[BZ77]{MR579172}
I.~N. Bernstein and A.~V. Zelevinsky, \emph{Induced representations of
  reductive {$p$}-adic groups. {I}}, Ann. Sci. \'{E}cole Norm. Sup. (4)
  \textbf{10} (1977), no.~4, 441--472. \MR{579172}

\bibitem[Clo90]{clozel_motifs}
Laurent Clozel, \emph{Motifs et formes automorphes: applications du principe de
  fonctorialit\'{e}}, Automorphic forms, {S}himura varieties, and
  {$L$}-functions, {V}ol. {I} ({A}nn {A}rbor, {MI}, 1988), Perspect. Math.,
  vol.~10, Academic Press, Boston, MA, 1990, pp.~77--159. \MR{1044819}

\bibitem[Dat05]{MR2120114}
J.-F. Dat, \emph{{$v$}-tempered representations of {$p$}-adic groups. {I}.
  {$l$}-adic case}, Duke Math. J. \textbf{126} (2005), no.~3, 397--469.
  \MR{2120114}

\bibitem[Del73]{MR0349635}
P.~Deligne, \emph{Les constantes des \'{e}quations fonctionnelles des fonctions
  {$L$}}, Modular functions of one variable, {II} ({P}roc. {I}nternat. {S}ummer
  {S}chool, {U}niv. {A}ntwerp, {A}ntwerp, 1972), 1973, pp.~501--597. Lecture
  Notes in Math., Vol. 349. \MR{0349635}

\bibitem[DHKM24]{dat2024finiteness}
Jean-Fran\c~cois Dat, David Helm, Robert Kurinczuk, and Gilbert Moss,
  \emph{Finiteness for {H}ecke algebras of {$p$}-adic groups}, J. Amer. Math.
  Soc. \textbf{37} (2024), no.~3, 929--949. \MR{4736530}

\bibitem[EH14]{MR3250061}
Matthew Emerton and David Helm, \emph{The local {L}anglands correspondence for
  {${\rm GL}_n$} in families}, Ann. Sci. \'{E}c. Norm. Sup\'{e}r. (4)
  \textbf{47} (2014), no.~4, 655--722. \MR{3250061}

\bibitem[Fen24]{feng2024smith}
Tony Feng, \emph{Smith theory and cyclic base change functoriality}, Forum
  Math. Pi \textbf{12} (2024), Paper No. e1, 66. \MR{4689770}

\bibitem[GL17]{genestier2017chtoucas}
Alain Genestier and Vincent Lafforgue, \emph{Restricted chtoucas for groups
  r$\backslash$'eductive and param$\backslash$'etrization of local langlands},
  arXiv preprint arXiv:1709.00978 (2017).

\bibitem[Hel16a]{MR3508741}
David Helm, \emph{The {B}ernstein center of the category of smooth {$W(k)[{\rm
  GL}_n(F)]$}-modules}, Forum Math. Sigma \textbf{4} (2016), Paper No. e11, 98.
  \MR{3508741}

\bibitem[Hel16b]{Helm_integral_bernstein_center}
\bysame, \emph{Whittaker models and the integral {B}ernstein center for {${\rm
  GL}_n$}}, Duke Math. J. \textbf{165} (2016), no.~9, 1597--1628. \MR{3513570}

\bibitem[Hen00]{henniart_une_preuve}
Guy Henniart, \emph{Une preuve simple des conjectures de {L}anglands pour
  {${\rm GL}(n)$} sur un corps {$p$}-adique}, Invent. Math. \textbf{139}
  (2000), no.~2, 439--455. \MR{1738446}

\bibitem[Hen01]{henniart_bordeaux}
\bysame, \emph{Sur la conjecture de {L}anglands locale pour {${\rm GL}_n$}},
  vol.~13, 2001, 21st Journ\'{e}es Arithm\'{e}tiques (Rome, 2001),
  pp.~167--187. \MR{1838079}

\bibitem[HM18]{MR3867634}
David Helm and Gilbert Moss, \emph{Converse theorems and the local {L}anglands
  correspondence in families}, Invent. Math. \textbf{214} (2018), no.~2,
  999--1022. \MR{3867634}

\bibitem[HT01]{harris_taylor}
Michael Harris and Richard Taylor, \emph{The geometry and cohomology of some
  simple {S}himura varieties}, Annals of Mathematics Studies, vol. 151,
  Princeton University Press, Princeton, NJ, 2001, With an appendix by Vladimir
  G. Berkovich. \MR{1876802}

\bibitem[JPSS81]{MR620708}
H.~Jacquet, I.~I. Piatetski-Shapiro, and J.~Shalika, \emph{Conducteur des
  repr\'{e}sentations du groupe lin\'{e}aire}, Math. Ann. \textbf{256} (1981),
  no.~2, 199--214. \MR{620708}

\bibitem[KM17]{MR3595906}
Robert Kurinczuk and Nadir Matringe, \emph{Rankin-{S}elberg local factors
  modulo {$\ell$}}, Selecta Math. (N.S.) \textbf{23} (2017), no.~1, 767--811.
  \MR{3595906}

\bibitem[KM21]{MR4311563}
R.~Kurinczuk and N.~Matringe, \emph{The {$\ell$}-modular local {L}anglands
  correspondence and local constants}, J. Inst. Math. Jussieu \textbf{20}
  (2021), no.~5, 1585--1635. \MR{4311563}

\bibitem[MM22]{matringe2022kirillov}
Nadir Matringe and Gilbert Moss, \emph{The kirillov model in families},
  Monatshefte f{\"u}r Mathematik (2022), 1--18.

\bibitem[Mos16]{MR3556429}
Gilbert Moss, \emph{Gamma factors of pairs and a local converse theorem in
  families}, Int. Math. Res. Not. IMRN (2016), no.~16, 4903--4936. \MR{3556429}

\bibitem[Mos21]{MR4335904}
\bysame, \emph{Characterizing the mod-{$\ell$} local {L}anglands correspondence
  by nilpotent gamma factors}, Nagoya Math. J. \textbf{244} (2021), 119--135.
  \MR{4335904}

\bibitem[MS14]{MR3178433}
Alberto M\'{\i}nguez and Vincent S\'{e}cherre, \emph{Repr\'{e}sentations lisses
  modulo {$l$} de {${\rm GL}_m({D})$}}, Duke Math. J. \textbf{163} (2014),
  no.~4, 795--887. \MR{3178433}

\bibitem[Ron16]{MR3551160}
Niccol\`o Ronchetti, \emph{Local base change via {T}ate cohomology}, Represent.
  Theory \textbf{20} (2016), 263--294. \MR{3551160}

\bibitem[Sch13]{scholze_llc}
Peter Scholze, \emph{The local {L}anglands correspondence for {${\rm GL}_n$}
  over {$p$}-adic fields}, Invent. Math. \textbf{192} (2013), no.~3, 663--715.
  \MR{3049932}

\bibitem[Ser]{MR1324577}
Jean-Pierre Serre, \emph{Cohomologie galoisienne}, fifth ed., Lecture Notes in
  Mathematics, vol.~5.

\bibitem[TV16]{MR3432583}
David Treumann and Akshay Venkatesh, \emph{Functoriality, {S}mith theory, and
  the {B}rauer homomorphism}, Ann. of Math. (2) \textbf{183} (2016), no.~1,
  177--228. \MR{3432583}

\bibitem[Vig96]{MR1395151}
Marie-France Vign\'{e}ras, \emph{Repr\'{e}sentations {$l$}-modulaires d'un
  groupe r\'{e}ductif {$p$}-adique avec {$l\ne p$}}, Progress in Mathematics,
  vol. 137, Birkh\"{a}user Boston, Inc., Boston, MA, 1996. \MR{1395151}

\bibitem[Vig98]{Vigneras_Induced}
\bysame, \emph{Induced {$R$}-representations of {$p$}-adic reductive groups},
  Selecta Math. (N.S.) \textbf{4} (1998), no.~4, 549--623. \MR{1668044}

\bibitem[Vig01]{MR1821157}
\bysame, \emph{Correspondance de {L}anglands semi-simple pour {${\rm GL}(n,F)$}
  modulo {$\ell\neq p$}}, Invent. Math. \textbf{144} (2001), no.~1, 177--223.
  \MR{1821157}

\bibitem[Vig04]{MR2058628}
\bysame, \emph{On highest {W}hittaker models and integral structures},
  Contributions to automorphic forms, geometry, and number theory, Johns
  Hopkins Univ. Press, Baltimore, MD, 2004, pp.~773--801. \MR{2058628}

\bibitem[Zel80]{MR584084}
A.~V. Zelevinsky, \emph{Induced representations of reductive {$p$}-adic groups.
  {II}. {O}n irreducible representations of {$ {\rm GL}(n)$}}, Ann. Sci.
  \'{E}cole Norm. Sup. (4) \textbf{13} (1980), no.~2, 165--210. \MR{584084}

\end{thebibliography}
	\noindent
	Santosh Nadimpalli, \\
	\texttt{nvrnsantosh@gmail.com}, \texttt{nsantosh@iitk.ac.in}.\\
	Sabyasachi Dhar,\\
	\texttt{sabya@iitk.ac.in}\\
	Department of Mathematics and Statistics, Indian
	Institute of Technology Kanpur, U.P. 208016, India.
\end{document}